\documentclass[a4paper,11pt, reqno, oneside]{amsart}

\usepackage[utf8]{inputenc}
\usepackage{amsmath, amsfonts, amssymb, amsthm, mathtools, bbm, enumerate}

\usepackage[all]{xy}
\usepackage[height=22.5cm, width=15.5cm]{geometry}
\usepackage[active]{srcltx}

\theoremstyle{plain}
\newtheorem{thm}{Theorem}[section]
\newtheorem{lemma}[thm]{Lemma}
\newtheorem{prop}[thm]{Proposition}
\newtheorem{cor}[thm]{Corollary}
\newtheorem{thm1}{Theorem}
\newtheorem{cor1}{Corollary}

\theoremstyle{definition}
\newtheorem{defi}[thm]{Definition}
\newtheorem{rmk}[thm]{Remark}
\newtheorem{ex}[thm]{Example}

\newcommand{\R}{\mathbb{R}}
\newcommand{\C}{\mathbb{C}}
\newcommand{\Z}{\mathbb{Z}}

\newcommand{\N}{\mathbb{N}}
\newcommand{\g}{\mathfrak{g}}

\newcommand{\id}{\mathbbm{1}}

\providecommand{\MR}{\relax\ifhmode\unskip\space\fi MR }
\providecommand{\href}[2]{#2}

\title[Globalizations of infinitesimal actions on supermanifolds]
{Globalizations of infinitesimal actions on supermanifolds}
\author{Hannah Bergner}
\address{Fakult\"at f\"ur Mathematik, Ruhr-Universit\"at Bochum, Universit\"atsstr. 150, D-44780 Bochum, Germany}
\email{Hannah.Bergner-C9q@rub.de}

\thanks{Financial support by SFB/TR 12 ``Symmetries and Universality in
Mesoscopic Systems'' of the DFG is gratefully acknowledged.}

\begin{document}
\begin{abstract}
 Let $\mathcal G$ be a Lie supergroup with Lie superalgebra $\g$, $\mathcal M$ a 
supermanifold and $\mathrm{Vec}(\mathcal M)$ the set of vector fields on 
$\mathcal M$. Let $\lambda:\g\rightarrow \mathrm{Vec}(\mathcal M)$ be an 
infinitesimal action, i.e. a homomorphism of Lie superalgebras.
 We show the existence of a local $\mathcal G$-action on $\mathcal M$ inducing 
the infinitesimal action $\lambda$ and find necessary and sufficient conditions 
for the existence of a 
 globalization in the sense of Palais.
\end{abstract}

\maketitle
\tableofcontents
\section{Introduction}
Let $\g$ be a finite-dimensional Lie superalgebra, $\mathcal G$ any Lie 
supergroup with $\g$ as its Lie algebra of right-invariant vector fields, 
$\mathcal M$ a supermanifold and denote the set of vector fields on $\mathcal M$ 
by $\mathrm{Vec}(\mathcal M)$.
Any action $\varphi:\mathcal G\times\mathcal M\rightarrow \mathcal M$, or local 
action, of the Lie supergroup $\mathcal G$ on $\mathcal M$
induces an infinitesimal action on $\mathcal M$, i.e. a homomorphism of Lie 
superalgebra $\lambda: \g\rightarrow\mathrm{Vec}(\mathcal M)$,
by setting 
$$\lambda(X)=(X(e)\otimes \mathrm{id}_\mathcal M^*)\circ\varphi^*.$$
The vector field $(X\otimes \mathrm{id}_\mathcal M^*)$ denotes the 
extension of the right-invariant vector field $X$ on $\mathcal G$ to a vector 
field 
on $\mathcal G\times\mathcal M$, and 
$(X(e)\otimes \mathrm{id}_\mathcal M^*)$ is its evaluation in the neutral 
element $e$ of $\mathcal G$.
Starting with an infinitesimal action $\lambda: 
\g\rightarrow\mathrm{Vec}(\mathcal M)$ of $\mathcal G$ on $\mathcal M$, 
it is a natural question to ask in which cases this infinitesimal action is 
induced, in the just describes way, by a local or global action of $\mathcal G$ 
on $\mathcal M$, or some
larger supermanifold $\mathcal M^*$ containing $\mathcal M$ as an open 
subsupermanifold.

In the case of a smooth manifold $M$ and a Lie group $G$, Palais studied these 
questions in detail (\cite{Palais}).
Concerning the existence of local actions, he showed that every infinitesimal 
action $\lambda$ is induced by a local action of $G$ on $M$.
This generalizes the fact that the flow of any vector field on $M$ defines a 
local $\R$-action on $M$.

In the case of one (not necessarily homogeneous) vector field $X$ on a 
supermanifold,
Monterde and S\'anchez-Valenzuela, and Garnier and Wurzbacher proved that the 
flow 
$\varphi:\mathcal W\subseteq \R^{1|1}\times\mathcal M\rightarrow\mathcal M$ of 
$X$ defines a local $\R^{1|1}$-action on $\mathcal M$
if and only if $X$ is contained in a $1|1$-dimensional Lie subsuperalgebra $\g$ 
of $\mathrm{Vec}(\mathcal M)$
(see \cite{MonterdeSanchez} and \cite{GarnierWurzbacher}). In 
\cite{GarnierWurzbacher} the same is also shown for a holomorphic vector field 
on 
a complex supermanifold.

In \cite{Palais}, Palais also found necessary and sufficient conditions for the 
existence of a globalization,
i.e. a (possibly non-Hausdorff) manifold $M^*$, containing on $M$ as an open 
submanifold, with an $G$-action on $M^*$ that induces $\lambda$ and satisfies 
$G\cdot M=M^*$.

In this paper, we extend these results to the case of (real or complex) 
supermanifolds and (real or complex) Lie supergroups.
The existence of a local actions with a given infinitesimal actions and 
conditions for the existence of globalizations
are proven.
A key point in the proof is, as in the classical case in \cite{Palais}, the 
study of the distribution $\mathcal D=\mathcal D_\lambda$ on the product 
$\mathcal G\times\mathcal M$
spanned by vector fields of the form
$$X+\lambda(X)\text{  for  } X\in\g,$$
considering $X$ and $\lambda(X)$ as vector fields on the product $\mathcal 
G\times\mathcal M$. 
Also, the fact that the flow of one even vector field on a supermanifold defines 
a local $\R$-action, or $\C$-action in the complex case,
is used (see \cite{MonterdeSanchez} and \cite{GarnierWurzbacher}).

The outline of this paper is the following:\\
First, some notations are introduced and a few basic definitions are given. 
In Section~$\S 3$, facts about distributions on supermanifolds are collected.
Then, the relation between infinitesimal and local actions on supermanifolds is 
studied in $\S 4$. 
The main result there is the equivalence of infinitesimal and local actions up 
to shrinking:

\begin{thm1}
 Let $\lambda:\g\rightarrow \mathrm{Vec}(\mathcal M)$ be an infinitesimal 
action.
Then there exists a local $\mathcal G$-action $\varphi:\mathcal W\subseteq 
\mathcal G\times\mathcal M\rightarrow \mathcal M$ on $\mathcal M$ with induced 
infinitesimal action $\lambda$.

Moreover, any local action $\varphi:\mathcal W\rightarrow \mathcal M$ is 
uniquely determined by its induced infinitesimal action and domain of 
definition.
\end{thm1}

In Section $\S 5$, conditions for the existence of globalizations are studied. 
To generalize the classical result of Palais, the notion 
of univalence of an infinitesimal action $\lambda$ is extended to
the case of supermanifolds. 
In the case of an infinitesimal action $\lambda$ whose underlying action is 
globalizable, the obstruction for an $\lambda$ to be globalizable 
is a holonomy phenomenon. 
The main result for the conditions of globalizability is the following:

\begin{thm1}
 The infinitesimal action $\lambda:\g\rightarrow\mathrm{Vec}(\mathcal M)$ is 
globalizable if and only if one of the following equivalent conditions is 
satisfied:
 \begin{enumerate}[(i)]
  \item The restricted infinitesimal action 
$\lambda|_{\g_0}:\g_0\rightarrow\mathrm{Vec}(\mathcal M)$ is globalizable to an 
action of $G$, where $\g_0$ denotes the even part of $\g$.
  \item The infinitesimal action $\lambda$ is univalent.
  \item The underlying infinitesimal action is globalizable,
  and all leaves $\Sigma\subset G\times M$ of the distribution $\mathcal 
D_\lambda$ are ``holonomy free''.
 \end{enumerate}
\end{thm1}

Condition $(iii)$ of this theorem together with the appropriate classical results 
of Palais yields the following in Section $\S 6$:

\begin{cor1}\label{cor: Cor 1}
 Let $\mathcal G$ be a simply-connected Lie supergroup and 
$\lambda:\g\rightarrow \mathrm{Vec}(\mathcal M)$ an
 infinitesimal action whose support is relatively compact in $M$. 
 Then there exists a global $\mathcal G$-action on $\mathcal M$ which induces 
$\lambda$.
 
 In particular, there is a one-to-one correspondence between infinitesimal and 
global actions of a simply-connected Lie supergroup on 
supermanifold with compact underlying manifold.
 \end{cor1}

\begin{cor1}
 Let $\lambda:\g\rightarrow\mathrm{Vec}(\mathcal M)$ be an infinitesimal action 
of a simply-connected Lie supergroup $\mathcal G$
 and let $\{X_i\}_{i\in I}$, $X_i\in \g_0$ be a set of generators of $\g_0$ such 
that underlying vector fields of 
 $\lambda(X_i)$, $i\in I$, on $M$ have global flows.
 Then there exists a global $\mathcal G$-action on $\mathcal M$ which induces 
$\lambda$.
\end{cor1}

A special version of Corollary 2 in the context of DeWitt supermanifolds was 
published by Tuynman in \cite{Tuynman}.

I would like to thank Peter Heinzner for suggesting this topic and helpful discussions.

\section{Notation}
Throughout, we work with the ``Berezin-Leites-Kostant''-approach to 
supermanifolds (see, e.g. \cite{Berezin}, \cite{Leites}, or \cite{Kostant}). 
A supermanifold is denoted by a calligraphic letter, e.g. $\mathcal M$, and the 
underlying manifold by the corresponding standard uppercase letters, e.g. $M$. 
The structure sheaf of a supermanifold $\mathcal M$ shall be denoted by 
$\mathcal O_\mathcal M$. 
If not otherwise mentioned, the considered supermanifolds are assumed to be real 
supermanifolds, i.e. the structure sheaf is locally isomorphic to $\mathcal 
C_U^{\infty}\otimes\bigwedge \R^n$ 
for appropriate open subsets $U\subset M$, where $\mathcal C_U^\infty$ denotes 
the sheaf of smooth functions on $U$. 
By a complex supermanifold a supermanifold whose structure sheaf is locally 
isomorphic to $\mathcal O_U\otimes\bigwedge \C^n$ is meant, where $\mathcal O_U$ 
is the sheaf of holomorphic 
functions on $U$.

If $\varphi:\mathcal M\rightarrow \mathcal N$ is a morphism between 
supermanifolds $\mathcal M$ and $\mathcal N$, let $\tilde{\varphi}:M\rightarrow 
N$ be the underlying map and
$\varphi^*:\mathcal O_\mathcal N\rightarrow \tilde{\varphi}_*\mathcal O_\mathcal 
M$ its pullback, i.e. $\varphi=(\tilde{\varphi},\varphi^*)$.

A (smooth/holomorphic) vector field or derivation $X$ on a (real/complex) 
supermanifold $\mathcal M$ is a graded 
(real-/complex-linear) derivation of sheaves $X:\mathcal O_\mathcal M\rightarrow 
\mathcal O_\mathcal M$. 
The notation for the sheaf of derivations or tangent sheaf will be 
$\mathrm{Der}\,\mathcal O_\mathcal M$ or $\mathcal T_\mathcal M$, and
$\mathrm{Vec}(\mathcal M)=\mathrm{Der}\,\mathcal O_\mathcal M(M)$. %The set 
$\mathrm{Vec}(\mathcal M)$ is a Lie superalgebra, possibly of infinite 
dimension.

A (real/complex) Lie supergroup can be defined to be a group object in the 
category of (real/complex) supermanifolds, i.e. a (real/complex) supermanifold 
$\mathcal G=(G,\mathcal O_\mathcal G)$ together
with morphisms for multiplication, inversion and the neutral element such that 
the usual group axioms are satisfied. 
The underlying manifold $G$ is a classical Lie group. In the following, we 
always assume that $G$ is connected.
A vector field $X$ on a Lie supergroup $\mathcal G$ is called right-invariant if 
$\mu^*\circ X=(X\otimes \mathrm{id}_\mathcal G^*)\circ\mu^*$
\footnote{The vector field $(X\otimes\mathrm{id}_\mathcal G^*)$ is the extension 
of the vector field $X$ on 
$\mathcal G$ to a vector field on the product $\mathcal G\times\mathcal G$ such 
that $X=(X\otimes\mathrm{id}_\mathcal G^*)\circ\pi_1^*$ and 
$0=(X\otimes\mathrm{id}_\mathcal G^*)\circ\pi_2^*$ if $\pi_i:\mathcal 
G\times\mathcal G\rightarrow \mathcal G$, $i=1,2$, is the projection onto the 
$i$-th component.}, 
where $\mu:\mathcal G\times\mathcal G\rightarrow \mathcal G$ denotes the 
multiplication on $\mathcal G$.
We define the Lie superalgebra $\g$ of $\mathcal G$ to be the set of 
right-invariant vector fields on $\mathcal G$. Its even part 
$\g_0$ can be identified with the Lie algebra (of right-invariant vector fields) 
of $G$.

\begin{defi}
 A local action of a Lie supergroup $\mathcal G=(G,\mathcal O_\mathcal G)$ on a 
supermanifold $\mathcal M=(M,\mathcal O_\mathcal M)$ is a morphism
$\varphi:\mathcal W\rightarrow \mathcal M$, $\mathcal W=(W,\mathcal O_{\mathcal 
G\times \mathcal M}|_W)$, where $W$ is an open neighbourhood of 
$\{e\}\times M$ in 
$G\times M$ such that $W_p=\{g\in G|\,(g,p)\in W\}$ is connected for each $p\in 
M$, satisfying the action properties:
\begin{enumerate}[(i)]
 \item If $\iota_e:\mathcal M\rightarrow \{e\}\times \mathcal M\subset \mathcal 
G\times \mathcal M$ is the canonical inclusion, then $\varphi\circ 
\iota_e=\text{id}_\mathcal M$.
\item If $\mu:\mathcal G\times \mathcal G\rightarrow G$ is the multiplication on 
$\mathcal G$, then the equation 
      $$\varphi\circ(\mu\times \text{id}_\mathcal M)=\varphi\circ 
(\text{id}_\mathcal G\times \varphi)$$ holds on the open subsupermanifold of 
      $\mathcal G\times\mathcal G\times\mathcal M$ where both sides are defined.
\end{enumerate}
\end{defi}

\section{Distributions on supermanifolds}

\subsection{Commuting vector fields}

Let $\mu:\R^{m|n}\times\R^{m|n}\rightarrow \R^{m|n}$ denote the addition on 
$\R^{m|n}$ 
which is given by $((s,\sigma),(t,\tau))\mapsto (s+t,\sigma +\tau)$ in 
coordinates.
The goal of this section is the proof of the following result:.

\begin{thm}\label{thm: commuting vector fields and local actions}
 Let $\mathcal{M}=(M,\mathcal{O}_\mathcal{M})$ be a supermanifold, 
$X_1,\ldots,X_m$ be even and $Y_1,\ldots,Y_n$ odd vector fields on $\mathcal{M}$
 which all commute. Then, for any $p\in M$ there exists an open connected 
neighbourhood $U$ of $0$ in $\R^m$, an open neighbourhood $V$ 
of $p$ in $M$ and a morphism
$\varphi=(\tilde{\varphi},\varphi^*): \mathcal{U}\times \mathcal{V} 
\rightarrow\mathcal{M}$, where $\mathcal{U}=(U,\mathcal{O}_{\R^{m|n}}|_U)$ and 
$\mathcal{V}=(V,\mathcal{O}_\mathcal{M}|_V)$, such that:
\begin{enumerate}[(i)]
 \item The map $\varphi$ has the action property, i.e. we have $\varphi\circ 
\iota_0=\mathrm{id}_\mathcal{M} $ if $\iota_0:\mathcal M\hookrightarrow 
\{0\}\times \mathcal M
\subset \R^{m|n}\times\mathcal M$ denotes the canonical inclusion, % (in 
particular the evaluation in $(t,\tau)=(0,0)$ is the pullback $\iota_0^*$), 
and the equality 
$\varphi\circ(\mathrm{id}_{\R^{m|n}}\times 
\varphi)=\varphi\circ(\mu\times\mathrm{id}_\mathcal{M})$
holds on the open
subsupermanifold of $\R^{m|n}\times \R^{m|n}\times\mathcal{M}$ where both sides 
are defined.

\item If $(t,\tau)$ are coordinates on $\R^{m|n}$, then for 
 all $i=1,\ldots,m$, $j=1\ldots,n$ we have
$$\frac{\partial}{\partial t_i}\circ \varphi^*=\varphi^*\circ X_i
\ \ \ \text{ and }\ \ \ 
\frac{\partial}{\partial \tau_j}\circ \varphi^*=\varphi^*\circ Y_j.
\footnote{Considering $\frac{\partial}{\partial t_i}$ and 
$\frac{\partial}{\partial \tau_j}$ as vector fields on the 
product $\R^{m|n}\times\mathcal M$.}$$
\end{enumerate}
By replacing $\R^{m|n}$ by $\C^{m|n}$ an analogous result also holds true for a 
complex supermanifold $\mathcal M$ and holomorphic vector fields 
$X_1,\ldots,X_m$ and $Y_1,\ldots,Y_n$.
\end{thm}

\begin{rmk}
 Note that $\varphi$ locally defines a local action of $\R^{m|n}$ on $\mathcal 
M$ in the sense that
 $\varphi$ would be a local action if the assumption that $\{0\}\times \mathcal 
M$ is contained in the domain of definition was dropped.
\end{rmk}

The proof makes use of the flows of vector fields on supermanifolds.

\begin{defi}
Let $X$ be an even vector field on a supermanifold $\mathcal M$. 
A flow of $X$ (with respect to the initial condition $\mathrm{id}_ \mathcal 
M:\mathcal M\rightarrow\mathcal M$ and $t_0=0\in\R$)
is a morphism $\varphi=\varphi^X:\mathcal{W}\rightarrow \mathcal{M}$,
where $\mathcal{W}=(W,\mathcal{O}_{\R\times\mathcal{M}}|_W)$ and $W\subseteq 
\R\times M$ is an open neighbourhood of $\{0\}\times M$ such that 
$W_p=\{t\in\R|(t,p)\in W\}\subseteq\R$ 
is connected for each $p\in M$, with 
\begin{enumerate}[(i)]
\item $\varphi\circ\iota_0=\mathrm{id}_\mathcal{M}$, and
\item $\frac{\partial}{\partial t}\circ \varphi^*=\varphi^*\circ X$.
\end{enumerate}

The flow $\varphi:\mathcal W\rightarrow \mathcal M$ is called maximal if for any 
flow 
$\varphi':\mathcal W'=(W',\mathcal O_{\R\times \mathcal M}|_{W'})\rightarrow 
\mathcal M$ of $X$
we have $W'\subseteq W$ and $\varphi'=\varphi|_{W'}$.
\end{defi}

\begin{rmk}\label{rmk: reduced vector field}
 Any vector field $X$ on a supermanifold $\mathcal M$ induces a vector field 
$\tilde{X}$ on 
the underlying manifold $M$. The reduced vector field $\tilde X$ can be defined 
by $\tilde{X}(f)=\mathrm{ev}(X(F))$,
if $F$ is a function on $\mathcal M$ with $\mathrm{ev}(F)=\tilde{F}=f$,
where $\mathrm{ev}:\mathcal O_\mathcal M\rightarrow \mathcal C^\infty_M$ denotes 
the evaluation map.
\end{rmk}

As in the classical case, there exists a unique maximal flow of an even vector 
field on a supermanifold.

\begin{thm}[see \cite{MonterdeSanchez}, Theorem 3.5/3.6, or 
\cite{GarnierWurzbacher}, Theorem 2.3 and Theorem 3.4]\label{thm: flow of a 
vector field}
Let $X$ be an even vector field on a supermanifold $\mathcal{M}$. then there 
exists a unique maximal flow $\varphi$ of $X$.
The reduced map $\tilde{\varphi}:W\rightarrow M$ is then the unique maximal flow 
of the reduced vector field $\tilde{X}$ on $M$. 
Moreover, the flow $\varphi$ defines a local $\R$-action on $\mathcal{M}$.
\end{thm}

\begin{rmk}
For any even holomorphic vector field $X$ on a complex supermanifold~$\mathcal 
M$ there also exists a holomorphic flow 
$\varphi:\mathcal W\subseteq \C\times\mathcal M\rightarrow\mathcal M$ 
(see \cite{GarnierWurzbacher}, Theorem 5.4).
But there are differences to the real case in terms of the possible domains of 
definitions of the flow; 
see \cite{GarnierWurzbacher} for details and examples.

As in the classical case, a flow $\varphi:\mathcal W\rightarrow \mathcal M$ of 
an even holomorphic vector field may not satisfy the equation
$\varphi\circ(\mathrm{id}_\C\times \varphi)=\varphi\circ (\mu_\C\times 
\mathrm{id}_\mathcal M)$ on all of the open subsupermanifold of 
$\C\times\C\times\mathcal M$
on which both sides of the equation are defined, denoting the multiplication on 
$\C$ by $\mu_\C$, and thus not define
a local $\C$-action on its domain of definition. One obtains a local $\C$-action 
only after shrinking its
domain of definition in a suitable way; for example after shrinking $W$ such 
that each $W_p=\{z\in \C|\,(z,p)\in W\}\subseteq \C$ is convex.
\end{rmk}

\begin{lemma}
 Let $X$ and $Y$ be vector fields on $\mathcal{M}$ and let $Y$ be even with flow 
$\varphi^Y$. 
If $\iota_t:\mathcal M\hookrightarrow \{t\}\times\mathcal M\subset 
\R\times\mathcal M$ is the canonical inclusion, denote by $\varphi^Y_t$ the 
composition $\varphi^Y\circ \iota_t$.
Then $$[X,Y]=\left.\frac{\partial}{\partial t}\right|_0(\varphi^Y_t)_*X$$ 
for $(\varphi^Y_t)_*X=(\varphi^Y_{-t})^*\circ X\circ (\varphi^Y_t)^*$, 
$\left.\frac{\partial}{\partial t}\right|_s=\iota_s^*\circ 
\frac{\partial}{\partial t}$ for $s\in \R$.
\end{lemma}

\begin{proof}
The lemma can be proven in a very similar way as the anologous classical result 
(see e.g. \cite{KobayashiNomizu}, Proposition 1.9).
A key point is the Taylor expansion $(\varphi^Y_t)^*(f)=f+t\hat{g}_t$ with 
$\hat{g}_0=Y(f)$,
where $\hat{g}\in \mathcal O_{\R\times \mathcal M} (I\times V)$, $I\times 
V\subseteq \R\times M$ open, $\hat{g}_t(x)=\hat{g}(t,x)$.
\end{proof}

As a corollary we get, as in the classical case (cf. \cite{KobayashiNomizu}, 
Corollary 1.10 and 1.11):

\begin{cor}\label{cor: commuting flows}
 Under the assumptions of the above lemma, we have:
\begin{enumerate}[(i)]
 \item $\left.\frac{\partial}{\partial 
t}\right|_s(\varphi^Y_t)_*X=(\varphi^Y_s)_*[X,Y]$
\item If $[X,Y]=0$, then $(\varphi^Y_t)_*X=X$, i.e. $(\varphi^Y_t)^*\circ 
X=X\circ(\varphi^Y_t)^*$, for all $t$.
\item If $[X,Y]=0$ and $X$ is also even, then the flows of $X$ and $Y$ commute, 
i.e. $\varphi^X_t\circ\varphi^Y_s=\varphi^Y_s\circ \varphi^X_t$ for all $s$, 
$t$.
\end{enumerate}
\end{cor}

\begin{proof}[Proof of Theorem \ref{thm: commuting vector fields and local 
actions}]
Let $X_1,\ldots, X_m$ be even and $ Y_1,\ldots, Y_n$ odd vector fields on 
$\mathcal{M}$ which all commute. Furthermore,  let 
$\varphi^{X_i}:\mathcal{W}_i\rightarrow \mathcal{M}$ denote
the flow of $X_i$ for any $i$. By Corollary~\ref{cor: commuting flows}~$(iii)$ 
these flows all commute.
Given $p\in M$, there exist an open neighbourhood $V\subseteq M$ of $p$ and an 
open connected neighbourhood $U\subseteq \R^m$ of $0$ such that 
$$\beta:U\times \mathcal{V}\rightarrow \mathcal{M},\, 
\beta_t=\varphi^{X_1}_{t_1}\circ\ldots\circ\varphi^{X_m}_{t_m},$$
is defined, where $\mathcal{V}=(V,\mathcal{O}_\mathcal{M}|_V)$.
Since the flows $\varphi^{X_i}$ commute and each flow defines a local 
$\R$-action on $\mathcal{M}$, the map $\beta$ has the action property, i.e.
${\beta_0}^*=\iota_0^*\circ\beta^*=\text{id}_\mathcal{M}^*$ and  
$\beta_s\circ\beta_s=\beta_{s+t}$ for all $s,t$ such that both sides of the 
equation are defined.

Let $\tau_1,\ldots,\tau_n$ be coordinates on $\R^{0|n}$ and define 
$$\alpha:\R^{0|n}\times \mathcal{M}\rightarrow \mathcal{M}  \text{ by 
}\alpha^*(f)=\exp \bigg(\sum_{j=1}^n \tau_j Y_j\bigg)(f).
\footnotemark$$ The underlying map is $\tilde{\alpha}=\text{id}_M$.
The sum $\exp(\sum_{j=1}^n \tau_j Y_j)(f)=\sum_{k=0}^\infty \frac 1 {k!} 
(\sum_{j=1}^n \tau_j Y_j)^k (f)$ is finite because $(\sum_{j=1}^n \tau_j 
Y_j)^{n+1}=0$.
\footnotetext{
The sum $\exp(\sum_{j=1}^n \tau_j Y_j)(f)=\sum_{k=0}^\infty \frac 1 {n!} 
(\sum_{j=1}^n \tau_j Y_j)^k (f)$ is to be understood in the following way: 
The vector fields $Y_j$, which are a priori vector fields on $\mathcal{M}$, are 
considered as vector fields on the product $\R^{0|n}\times \mathcal M $ and 
similarly $f$ is  considered as a function on this product. Hence, $\tau_j 
Y_j(f)$ here in fact means 
$\tau_j\cdot (\text{id}_{\R^{0|n}}^*\otimes Y_j)(  \pi_ \mathcal M^*(f))$, 
where $\tau_j$ is now considered as a coordinate on $\R^{0|n}\times 
\mathcal{M}$, 
$\pi_\mathcal M:\R^{0|n}\times \mathcal{M}\rightarrow \mathcal{M}$ is the 
canonical projection onto $\mathcal{M}$ and $\text{id}_{\R^{0|n}}^*\otimes Y_j$ 
is the extension of $Y_j$ to a vector field on $\R^{0|n}\times \mathcal{M}$.
}
Since the odd vector fields $Y_j$ all commute, i.e. $[Y_i,Y_j]=Y_iY_j+Y_jY_i=0$, 
we get 
$(\sigma_i Y_i)(\tau_j Y_j)=(\tau_j Y_j)(\sigma_i Y_i),$ 
and thus  
$\exp(\sum_{i} \sigma_i Y_i)\exp(\sum_{j} \tau_j Y_j)=\exp(\sum_{k} 
(\sigma_k+\tau_k) Y_k)$
if $(\sigma_1,\ldots,\sigma_n,\tau_1,\ldots,\tau_n)$ are coordinates on 
$\R^{0|n}\times \R^{0|n}$.
Consequently, the map $\alpha$ defines an $\R^{0|n}$-action on $\mathcal{M}$.

Now, let $\mathcal{W}=(U\times V,\mathcal{O}_{\R^{m|n}\times 
\mathcal{M}}|_{U\times V})$ and define 
$$\varphi:\mathcal{W}\rightarrow\mathcal{M},\,\varphi=\beta\circ(\text{id}_{\R^m
}\times \alpha)\text{ on }\mathcal{W}.$$
Then $\varphi^*(f)=\exp(\sum_{j}\tau_j 
Y_j)(\varphi^{X_1}_{t_1})^*\circ\ldots\circ(\varphi^{X_m}_{t_m})^*(f)$. 
The map $\varphi$ satisfies $\varphi\circ \iota_0=\mathrm{id}_\mathcal{M}$ and 
$\frac{\partial}{\partial t_i}\circ\varphi^*
=X_i\circ\varphi^*,
 \frac{\partial}{\partial \tau_j}\circ\varphi^*=Y_j\circ\varphi^*$ for all 
$i,j$,
making use of the commutativity of the vector fields and Corollary \ref{cor: 
commuting flows}.
A calculation using the action properties of $\alpha$ and $\beta$ and again 
Corollary~\ref{cor: commuting flows} shows that
$\varphi\circ(\text{id}_{\R^{m|n}}\times \varphi)
=\varphi\circ(\mu\times\text{id}_\mathcal{M})
$
on the open subsupermanifold of $\R^{m|n}\times\R^{m|n}\times \mathcal{M}$ on 
which both $\varphi\circ(\text{id}_{\R^{m|n}}\times \varphi)$ and
$\varphi\circ(\mu\times\text{id}_\mathcal{M})$ are defined.
\end{proof}
\begin{rmk}
 The complex version of the result can be proven along the lines of the real 
case using holomorphic flows and an analogue of Corollary \ref{cor: commuting 
flows} for holomorphic vector fields.
\end{rmk}

\subsection{Distributions and Frobenius theorem}

\begin{defi}
 A (smooth or holomorphic) distribution $\mathcal{D}$ on a (real or complex) 
supermanifold $\mathcal{M}$ is 
 a graded $\mathcal{O}_\mathcal{M}$-subsheaf 
of the tangent sheaf 
$\mathcal{T}_\mathcal{M}=\text{Der}\,\mathcal{O}_\mathcal{M}$ of $\mathcal{M}$ 
which is locally a direct factor, 
i.e. for each point $p\in M$ there exists an open neighbourhood $U$ of $p$ in 
$M$ and a subsheaf $\mathcal{E}$ of $\text{Der}\,\mathcal{O}_\mathcal{M}|_U$ on 
$U$ such that 
$\mathcal{D}|_U\oplus  \mathcal{E}=\text{Der}\,\mathcal{O}_\mathcal{M}|_U$.
\end{defi}

\begin{rmk}[cf. \cite{Varadarajan}, Section 4.7]
Locally there exist independent vector fields $X_1,\ldots,$ $X_r,$% 
$Y_1,\ldots ,Y_s$ spanning a distribution $\mathcal{D}$, where the $X_i$ are 
even and the $Y_j$ are odd. 
Moreover, $(r|s)=\dim \mathcal{D}(p)$ for any $p\in M$ (if $M$ is connected) and 
$(r|s)$ is called the rank of the distribution $\mathcal{D}$.

If $\psi:\mathcal M\rightarrow \mathcal N$ is a diffeomorphism and $\mathcal D$ 
is a distribution on $\mathcal M$, then there is a distribution 
$\psi_*(\mathcal D)$ on $\mathcal N$ which is spanned by vector fields of the 
form $\psi_*(X)=(\psi^{-1})^*\circ X\circ \psi^*$ for vector fields $X$ on 
$\mathcal M$ belonging 
to $\mathcal D$.
\end{rmk}

\begin{rmk}\label{rmk: reduced distribution}
A distribution $\mathcal D$ of rank $(r|s)$ on a supermanifold $\mathcal M$ 
induces a distribution 
$\widetilde{\mathcal D}$ of rank $r$ on the underlying classical manifold $M$ by 
defining 
$\widetilde{\mathcal D}(p)=\{\tilde X (p)|\,X\in\mathcal D\}\subseteq T_pM$ for 
$p\in M$. 
A vector field on $M$ belongs to the reduced distribution $\widetilde{\mathcal 
D}$ if and only if it is the reduced vector field $\tilde X$ of some 
vector field $X$ on $\mathcal M$ belonging to $\mathcal D$. 
\end{rmk}

As in the classical case, one can define the notion of an involutive 
distribution.
\begin{defi}
 A distribution $\mathcal{D}$ is called involutive if $\mathcal{D}_p$ is a Lie 
subsuperalgebra of $(\text{Der}\,\mathcal{O}_\mathcal{M})_p$ for each $p\in M$, 
i.e. if for any
two vector fields $X$ and $Y$ on $\mathcal{M}$ belonging to $\mathcal{D}$ their 
commutator $[X,Y]$ also belongs to $\mathcal{D}$.
\end{defi}

The local structure of an involutive distribution on a supermanifold is 
described by the following version of Frobenius theorem.

\begin{thm}[Local Frobenius Theorem, cf. \cite{Varadarajan} Theorem 
4.7.1]\label{thm: local Frobenius}
 A distribution $\mathcal{D}$ on a real (resp. complex) supermanifold 
$\mathcal{M}$ is involutive if and only if there are local coordinates 
$(x, \theta)=(x_1,\ldots,x_m,\theta_1,\ldots,\theta_n)$ around each point $p\in 
M$ such that $\mathcal{D}$ is locally spanned by $\frac{\partial}{\partial 
x_1},\ldots,\frac{\partial}{\partial x_r},
\frac{\partial}{\partial \theta_1},\ldots,\frac{\partial}{\partial \theta_s}$.
\end{thm}

If $\mathcal D$ is locally spanned by $\frac{\partial}{\partial 
x_1},\ldots,\frac{\partial}{\partial x_r},
\frac{\partial}{\partial \theta_1},\ldots,\frac{\partial}{\partial \theta_s}$, 
the involutiveness of $\mathcal{D}$ follows directly. 
A key point in the proof of the local Frobenius theorem is the existence of 
homogeneous 
commuting vector fields which locally span the distribution. Their existence can 
be proven as in the classical case (as for example in \cite{Lee}, Theorem 
19.10). Then the desired coordinates can be found by 
applying Theorem~\ref{thm: commuting vector fields and local actions}.

\begin{defi}
 Let $\mathcal{D}$ be a distribution of rank $(r|s)$ on a supermanifold 
$\mathcal{M}=(M,\mathcal{O}_\mathcal{M})$.
An $(r|s)$-dimensional subsupermanifold $j:\mathcal{N}\rightarrow \mathcal{M}$, 
$\mathcal{N}=(N,\mathcal{O}_\mathcal{N})$, of $\mathcal{M}$ is called an 
integral manifold of $\mathcal{D}$ 
through $p\in M$ if
\begin{enumerate}[(i)]
 \item the point $p$ is contained in $\tilde{\jmath}(N)$, and
\item the distribution $\mathcal D$ is tangent to $\mathcal N$, i.e. for any 
vector field $X$ on $\mathcal{M}$ belonging to $\mathcal{D}$ there exists a 
vector field $\bar{X}$ on $\mathcal{N}$ such that $\bar{X}\circ j^*=j^*\circ X$, 
or equivalently
$X(\ker j^*)\subseteq \ker j^*$, 
and all vector fields on $\mathcal{N}$ arise in this way. 
\end{enumerate}
\end{defi}

\begin{rmk}
 In the case of supermanifolds, integral manifolds of a distribution do not 
provide as much information about the distribution as in the classical case.

Contrary to the classical case, there is no global version of a Frobenius 
theorem for the above defined notion of an integral manifold; the existence of 
integral manifolds through every
point is not equivalent to the involutiveness of a distribution. 
Nevertheless, the local Frobenius theorem still guarantees the existence of 
integral manifolds through every point for an involutive distribution.
\end{rmk}

\begin{ex}[A non-involutive distribution with integral manifolds]\label{ex: 
non-involutive distribution}
 Let $\mathcal{M}=\R^{0|2}$, with coordinates $\theta_1$ and $\theta_2 $, and 
let $\mathcal{D}$ be the distribution on $\mathcal{M}$ 
spanned by the odd vector field $$X= \frac{\partial}{\partial\theta_1} 
+\theta_1\theta_ 2\frac{\partial}{\partial \theta_2}.$$
The distribution $\mathcal{D}$ is not involutive since 
$[X,X]=2XX=2\theta_2\frac{\partial}{\partial\theta_2}\notin\mathcal{D}$,
but $\mathcal{N}=\R^{0|1}\times\{0\}\subset\mathcal{M}$ is an integral manifold 
of $\mathcal D$, and thus there exist integral manifolds through each point 
(which is only $0$).
Moreover, remark that the involutive distribution spanned by 
$\frac{\partial}{\partial \theta_1}$ has the same integral manifolds as 
$\mathcal{D}$.
\end{ex}

\section{Infinitesimal and local group actions}
In the classical case, the flow $\varphi^X:\Omega\subseteq \R\times M\rightarrow 
M$ of a vector field~$X$ on a manifold $M$ defines a local $\R$-action on $M$.
More generally, as proven in \cite{Palais}, Chapter II, if we have a Lie algebra 
homomorphism $\lambda:\g\rightarrow \text{Vec}(M)$ 
of a finite dimensional Lie algebra $\g$ into the 
Lie algebra of vector fields on $M$, there is a local action 
$\varphi:\Omega\subseteq G\times M\rightarrow M$ of a Lie group $G$, 
with Lie algebra of right-invariant vector fields $\g$, such that
its induced infinitesimal action 
$$\g\rightarrow \text{Vec}(M),\,X\mapsto \left.\frac{\partial}{\partial 
t}\right|_0\varphi(\exp(tX),-)=(X(e)\otimes \text{id}_{M}^*)\circ \varphi^*$$ 
coincides with $\lambda$. A typical example is the case where $X_1,\ldots,X_n$ 
are vector fields on $M$ whose
$\R$-span $\g=\text{span}_\R\{X_1,\ldots,X_n\}$  is a Lie subalgebra of 
$\text{Vec}(M)$ and $\lambda$ is the inclusion $\g\hookrightarrow 
\text{Vec}(M)$.

The goal in this section is the proof of an analogous theorem on supermanifolds. 
For that purpose a suitable distribution on $\mathcal G\times\mathcal M$ is 
introduced, as in the classical case in \cite{Palais}.

This also generalizes the result in \cite{MonterdeSanchez} and 
\cite{GarnierWurzbacher} that the flow of one vector field $X$ on a 
supermanifold $\mathcal M$  
is a local $\R^{1|1}$-action if and only if $X$ is contained in a 
$1|1$-dimensional Lie subsuperalgebra $\g\subset \mathrm{Vec}(\mathcal M)$.

The results in this section are formulated for the real case, but are equally 
true in the complex case, i.e. for complex supermanifolds $\mathcal M$, 
complex Lie supergroups $\mathcal G$ and morphisms $\lambda:\g\rightarrow 
\mathrm{Vec}(\mathcal M)$ of complex Lie superalgebras.

\begin{prop}
 Let $\varphi$ be a local action of a Lie supergroup $\mathcal G$, with Lie 
superalgebra $\g$, on a supermanifold $\mathcal M$. 
Then $(X(e)\otimes \mathrm{id}_\mathcal M^*)\circ \varphi^* $ is a vector field 
on $\mathcal M$ for any $X\in \g$.\footnotemark
The map $$\lambda_\varphi:\g\rightarrow \mathrm{Vec}(\mathcal 
M)=\mathrm{Der}\,\mathcal O_\mathcal M(M),\ X\mapsto (X(e)\otimes 
\mathrm{id}_\mathcal M^*)\circ\varphi^*$$
is a homomorphism of Lie superalgebras.

\footnotetext{Here, and in the following, $(X\otimes \mathrm{id}_\mathcal M^*)$ 
denotes again the canonical 
extension of the vector field $X$ on
$\mathcal G$ to a vector field on $\mathcal G\times \mathcal M$,
and $(X(e)\otimes\mathrm{id}_\mathcal M^*)$ is its evaluation in $e$, i.e. 
$(X(e)\otimes\mathrm{id}_\mathcal M^*)=\iota_e^*\circ (X\otimes 
\mathrm{id}_\mathcal M^*)$.}

\end{prop}

\begin{proof}
Let $X,Y\in\g$ be homogeneous.
The vector field $X\otimes \mathrm{id}_\mathcal M^*$ has the parity 
$\left|X\right|$ of $X$.
Let $f,g\in \mathcal O_\mathcal M(M)$ and let $f$ be homogeneous with parity 
$\left|f\right|$. 
 Then, using $\varphi\circ\iota_e=\mathrm{id}_\mathcal M$, we get
\begin{align*}
&((X(e)\otimes\mathrm{id}_\mathcal 
M^*)\circ\varphi^*)(fg)=\iota_e^*(X\otimes\mathrm{id}_\mathcal 
M^*)(\varphi^*(f)\varphi^*(g))\\
&= \Big(((X(e)\otimes \mathrm{id}_\mathcal M^*)\circ\varphi^*)(f)\Big)g
+(-1)^{\left|X\right|\left|f\right|}f\Big(((X(e)\otimes\mathrm{id}_\mathcal 
M^*)\circ\varphi^*)(g)\Big).
\end{align*}
Since $\varphi$ is a local action, we have 
$\varphi\circ(\mu\times\mathrm{id}_\mathcal M)=\varphi\circ 
(\mathrm{id}_\mathcal G\times\varphi)$.
 A calculation using this identity then gives
$\lambda_\varphi(X)\lambda_\varphi(Y)=(XY(e)\otimes\id)\circ\varphi^*.$
Thus 
$$[\lambda_\varphi(X),\lambda_\varphi(Y)]=\left(\Big(XY(e)-(-1)^{
\left|X\right|\left|Y\right|}
YX(e)\Big)\otimes\id\right)\circ\varphi^*=\lambda_\varphi([X,Y]).$$
\end{proof}

\begin{defi}
Let $\mathcal M$ be a supermanifold and $\mathcal G$ a Lie supergroup with Lie 
superalgebra~$\g$.
An infinitesimal action of $\mathcal G$ on $\mathcal M$ is a  homomorphism 
$\lambda:\g\rightarrow \mathrm{Vec}(\mathcal M)$ of Lie superalgebras.

The induced infinitesimal action of a local $\mathcal G$-action~$\varphi$ on 
$\mathcal M$ is the homomorphism
$$\lambda_\varphi:\g\rightarrow \text{Vec}(\mathcal M 
),\,\lambda_\varphi(X)=(X(e)\otimes \text{id}_\mathcal M^*)\circ \varphi^*.$$
\end{defi}

As in the classical case, there is an equivalence of infinitesimal actions and 
local actions up to shrinking. This is is the content of the following theorem, 
whose proof is carried out
in the remainder of this section.
\begin{thm} 
Let $\lambda:\g\rightarrow \mathrm{Vec}(\mathcal M)$ be an infinitesimal action.
Then there exists a local $\mathcal G$-action $\varphi:\mathcal W\subseteq 
\mathcal G\times\mathcal M\rightarrow \mathcal M$ on $\mathcal M$ such that its 
induced infinitesimal action
$\lambda_\varphi$ equals $\lambda$.

Moreover, any local action $\varphi:\mathcal W\rightarrow \mathcal M$ is 
uniquely determined by its induced infinitesimal action and domain of 
definition.
\end{thm}

\subsection{Distributions associated to infinitesimal actions}
Given an infinitesimal action $\lambda:\g\rightarrow \text{Vec}(\mathcal M)$, 
define
the distribution $\mathcal D=\mathcal D_\lambda$ on $\mathcal G\times \mathcal 
M$ as the distribution spanned by vector fields of the form $X+\lambda(X)$%
\footnote{The vector field $X+\lambda(X)$ is here considered as a vector field 
on $\mathcal G\times \mathcal M$, so more formally one should write 
$X\otimes\text{id}_\mathcal M^*+\text{id}_\mathcal G^*\otimes \lambda(X)$ for 
$X+\lambda(X)$.}
for $X\in\g$ (cf. \cite{Palais}, Chapter II, Definition 7). 
The rank of the distribution $\mathcal D$ equals the dimension of the Lie 
superalgebra $\g$.

If the homomorphism $\lambda:\g\rightarrow \text{Vec}(\mathcal M)$ is given, we 
can take 
$\mathcal G$ to be any Lie supergroup with Lie superalgebra $\g$. One choice is 
for example the unique Lie supergroup with simply-connected underlying Lie group
and Lie superalgebra $\g$ (for the existence of such $\mathcal G$ see, e.g., 
\cite{Koszul}, and in the complex case \cite{Vishnyakova}).

In the following, properties of distributions associated to an infinitesimal 
action are studied.

\begin{lemma}
The distribution $\mathcal D=\mathcal D_\lambda$ associated to an infinitesimal 
action $\lambda$ is involutive.
\end{lemma}
\begin{proof}
Since $\mathcal D$ is spanned by vector fields of the form $X+\lambda(X)$ for 
$X\in\g$, it is enough to check that
the bracket of two such vector fields belongs again to $\mathcal D$. 
Using that $\lambda$ is a homomorphism of Lie superalgebras we get
$$[X+\lambda(X),Y+\lambda(Y)]=[X,Y]+[\lambda(X),\lambda(Y)]=[X,Y]+\lambda([X,Y]
)$$ for any $X,Y\in\g$  and thus 
$[X+\lambda(X),Y+\lambda(Y)]=[X,Y]+\lambda([X,Y])$ also belongs to $\mathcal D$.
\end{proof}

The local Frobenius theorem now yields that a distribution $\mathcal D$ 
associated to an infinitesimal action locally looks like a standard distribution 
on a product.
In the following, local charts for the distribution $\mathcal D$ which locally 
transform $\mathcal D$ to a standard distribution and satisfy a few more 
properties with 
respect to the product structure of $\mathcal G\times \mathcal M$ are of special 
interest.

\begin{defi}[flat chart]
 Let $\mathcal G$ be a Lie supergroup, $\mathcal M$ a supermanifold and 
$\mathcal D$ the distribution on 
$\mathcal G\times \mathcal M$ associated to an infinitesimal action on $\mathcal 
M$.
Let $U\subseteq G$ be an open connected neighbourhood of a point $g\in G$, 
$\mathcal U=(U,\mathcal O_\mathcal G|_U)$, and denote by 
$\iota_g:\mathcal M\hookrightarrow \{g\}\times \mathcal M\subset \mathcal 
G\times \mathcal M$ the canonical inclusion.
Moreover, let $V\subseteq M$ be open, $\mathcal V=(V,\mathcal O_\mathcal M|_V)$, 
and let $\rho:\mathcal V\rightarrow \mathcal M$ be a diffeomorphism onto its 
image.
Denote by $\mathcal D_\mathcal G$ the standard distribution on $\mathcal G\times 
\mathcal M$ in $\mathcal G$-direction, which is spanned by vector fields 
$X\otimes \mathrm{id}_\mathcal M^*$ on $\mathcal G\times\mathcal M$, where $X$ 
is an arbitrary vector field
on $\mathcal G$.

A diffeomorphism onto its image $\psi:\mathcal U\times \mathcal V\rightarrow 
\mathcal G\times \mathcal M$ is called a flat chart with respect to
$(\mathcal D,U,V,g,\rho)$, or simply a flat chart (in $g$), if the following 
conditions are satisfied:
\begin{enumerate}[(i)]
 \item $\psi_*(\mathcal D_\mathcal G|_W)=\mathcal D|_{\tilde{\psi}(W)}$ for each 
open subset $W\subseteq U\times V$
 \item $\pi_\mathcal G|_{U\times V}=\pi_\mathcal G\circ\psi$ for the projection 
$\pi_\mathcal G:\mathcal G\times \mathcal M\rightarrow \mathcal G$
 \item $\psi\circ\iota_g|_V=\iota_g|_{\tilde{\rho}(V)}\circ \rho$
\end{enumerate}
\end{defi}

\begin{rmk}\label{rmk: flat chart}
If $\psi:\mathcal U\times \mathcal V_1\rightarrow \mathcal G\times \mathcal M$ 
is a flat chart with respect to $(\mathcal D, U, V_1, g, \rho_1)$ and 
$\rho_2:\mathcal V_2\rightarrow\mathcal M$ is 
 diffeomorphism onto its image with $\tilde{\rho}_2(V_2)\subseteq V_1$, then the 
map
$\psi'=\psi\circ(\mathrm{id}_\mathcal U\times  \rho_2)$ is a flat chart with 
respect to $(\mathcal D,U,V_2,g,\rho_1\circ \rho_2)$.
\end{rmk}

\begin{lemma}\label{lemma: even part of a flat chart}
 Let $\mathcal D$ be the distribution on $\mathcal G\times \mathcal M$ 
associated to the infinitesimal action 
$\lambda:\g\rightarrow\mathrm{Vec}(\mathcal M)$, 
$\lambda_0=\lambda|_{\g_0}:\g_0\rightarrow \mathrm{Vec}(\mathcal M)$ 
the restriction of $\lambda$ to the even part $\g_0$ of $\g$, 
and $\mathcal D_0$ the distribution on $G\times\mathcal M$ associated to 
$\lambda_0$.
Let $\psi:\mathcal U\times\mathcal V\rightarrow \mathcal G\times\mathcal M$ be a 
flat chart with 
respect to $(\mathcal D, U, V, g, \rho)$
and define 
$\psi_0:U\times \mathcal V\rightarrow G\times \mathcal M$ by
$\psi_0=(\mathrm{id}_G\times \varphi_0)\circ 
(\mathrm{diag}\times\mathrm{id}_\mathcal V)$,
where $\mathrm{diag}:U\rightarrow  U\times U$ denotes the diagonal and 
$\varphi_0$ is the composition of $\pi_\mathcal M\circ \psi$ and the canonical 
inclusion 
$U\times \mathcal V\hookrightarrow\mathcal U\times \mathcal V$.
Then $\psi_0$ is a flat chart with respect to $(\mathcal D_0,U, V, g,\rho)$.
\end{lemma}

\begin{proof}
 It can be checked by direct calculations that $\psi_0$ is a flat chart,
 using that the even right-invariant vector fields on $\mathcal G$ can 
 be identified with the right-invariant vector fields on $G$, i.e. 
$\mathrm{Lie}(G)\cong \g_0$
 if $\g_0$ denotes the even part of the Lie superalgebra $\g=\g_0+\g_1$ of 
$\mathcal G$.
\end{proof}

\begin{prop}[Local existence of flat charts]\label{prop: existence of flat 
charts}
 Let $\mathcal D$ be the distribution associated to the infinitesimal action 
$\lambda:\g\rightarrow \mathrm{Vec}(\mathcal M)$ on the supermanifold $\mathcal 
M$ 
and let $\mathcal G$ be a Lie supergroup with $\g$ as its Lie superalgebra of 
right-invariant vector fields.
For any point $(g,p)\in G\times M$ there are an open connected neighbourhood $U$ 
of $g$ in $G$ and an open neighbourhood $V$ of~$p$ in $M$ such that there exists
a flat chart $\psi:\mathcal U\times \mathcal V\rightarrow \mathcal G\times 
\mathcal M$ with respect to $(\mathcal D,U,V,g,\rho=\mathrm{id}_\mathcal V)$. 

By the above remark this implies moreover
the existence of $U$ and $V$ and a flat chart with respect to $(\mathcal D, U, 
V, g, \rho)$ for arbitrary $\rho$.
\end{prop}
\begin{proof}
 Let $X_1,\ldots,X_k,Y_1,\ldots,Y_l$ be a basis of $\g$ such that 
$X_1,\ldots,X_k$ are even and $Y_1,\ldots,Y_l$ odd vector fields. 
Then the tangent vectors $X_1(g'),\ldots,X_k(g')$, $Y_1(g'),\ldots ,Y_l(g')\in 
T_{g'}\mathcal G=\{X(g')|\,X\in\g\}$ are linearly independent for all $g'\in 
G$. 
Since the distribution $\mathcal D$ is spanned by vector fields of the form 
$X+\lambda(X)$ for $X\in \g$, there exist local coordinates $(t,\tau)$ for 
$\mathcal G$  on an open connected neighbourhood $U\subseteq G$ 
of $g$ and local coordinates $(x,\theta)$ for $\mathcal M$ on an open 
neighbourhood  $V\subseteq M$ of $p$ so that 
$\mathcal D$ is locally spanned by the commuting vector fields
$$A_i=\frac{\partial}{\partial t_i}+\sum_{u=1}^m a_{iu}\frac{\partial}{\partial 
x_u}+\sum_{v=1}^n b_{iv}\frac{\partial}{\partial \theta_v} 
\ \ \ \ \mathrm{ and }\ \ \ \
B_j=\frac{\partial}{\partial \tau_j}+\sum_{u=1}^m c_{ju}\frac{\partial}{\partial 
x_u}+\sum_{v=1}^n d_{jv}\frac{\partial}{\partial \theta_v}$$
for $i=1\ldots,k$ and $j=1,\ldots,l$, where $a_{iu},b_{iv},c_{ju}, d_{jv}\in 
\mathcal O_{\mathcal G\times \mathcal M}(U\times V)$.

After shrinking $U$ and $V$, $\mathcal U=(U, \mathcal O_\mathcal G|_U)$ can be 
assumed to be an open subsupermanifold of $\R^{k|l}$ 
with $g=0$ and there is
a morphism $\varphi:\mathcal U\times (\mathcal U\times \mathcal V)\rightarrow 
\mathcal G\times \mathcal M$ 
associated to the above defined commuting vector fields, satisfying the
properties $\varphi\circ\iota_0=\mathrm{id}$ and 
$\varphi\circ(\mathrm{id}_{\R^{k|l}}\times 
\varphi)=\varphi\circ(\mu_{\R^{k|l}}\times \mathrm{id}_\mathcal M)$
(cf. Theorem \ref{thm: commuting vector fields and local actions}). 
Since $$\frac{\partial}{\partial t_i}\circ\varphi^*=
\varphi^*\circ A_i
\ \ \ \ \mathrm{ and }\ \ \ \
\frac{\partial}{\partial \tau_j}\circ\varphi^*=
\varphi^*\circ B_j,$$
the subsupermanifold $\mathcal V\cong\{0\}\times \mathcal V\subset \mathcal 
U\times \mathcal V$ is transversal to $\varphi(\mathcal U\times \{(0,p)\})$ in 
$\tilde{\varphi}(0,0,p)=(0,p)$.
The map $\psi:\mathcal U\times \mathcal V\rightarrow \mathcal G\times \mathcal 
M$, $\psi=\varphi|_{\mathcal U\times\{0\}\times \mathcal V}$. identifies the 
standard distribution $\mathcal D_\mathcal U$ on $\mathcal U\times \mathcal V$ 
with $\mathcal D$,
i.e. $\psi_*(\mathcal D_\mathcal G|_W)=\psi_*(\mathcal D_\mathcal U|_W)=\mathcal 
D|_{\tilde{\psi}(W)}$
 for all open subsets $W\subseteq U\times V$. 
The action property of the map $\varphi$ moreover implies that $$\psi\circ 
\iota_g|_V=\psi\circ\iota_0|_V=\varphi|_{\mathcal U\times\{0\}\times\mathcal 
V}\circ\iota_0|_V
=\mathrm{id}_{\mathcal U\times\mathcal V}|_{\{0\}\times 
V}=\iota_0|_V=\iota_g|_V.$$
Hence, it only remains to show that $\pi_\mathcal G|_{U\times V}=\pi_\mathcal 
G\circ\psi$ in order to prove that $\psi$ is a flat chart.
This is equivalent to showing $\psi^*(t_i)=t_i$ and $\psi^*(\tau_j)=\tau_j$ for 
all $i$, $j$, where $t_i$ and $\tau_j$ are now considered as local coordinate 
functions on 
$\mathcal G\times \mathcal M$. 
Since $\frac{\partial}{\partial 
t_r}(\psi^*(t_i))=\psi^*(\psi_*(\frac{\partial}{\partial 
t_r})(t_i))=\psi^*(0)=0$ for all $r\neq i$,
$\frac{\partial}{\partial t_i}(\psi^*(t_i))=1$ and 
$\frac{\partial}{\partial 
\tau_s}(\psi^*(t_i))=\psi^*(\psi_*(\frac{\partial}{\partial 
\tau_s})(t_i))=\psi^*(0)=0$ for all $s$, 
the function $\psi^*(t_i)\in\mathcal O_{\mathcal G\times\mathcal M}(U\times V)$ 
is of the form 
$$\psi^*(t_i)=(t_i+c_i(x))+\sum_{\nu\neq 0} c_\nu(x)\theta^\nu$$
for some smooth functions $c_i,c_\nu(\nu\neq 0)$  on  $V$.
The property $\psi\circ\iota_0=\iota_0$ now implies 
$0=\iota_0^*(t_i)=\iota_0^*\circ\psi^*(t_i)=(0+c_i(x))+\sum_{\nu\neq 0} 
c_\nu(x)\theta^\nu$ and therefore $c_i=c_\nu=0$. 
Hence $\psi^*(t_i)=t_i$ as required.
A similar argument yields $\psi^*(\tau_j)=\tau_j$.
\end{proof}

\begin{lemma}\label{lemma: almost flat chart}
 Let $\psi:\mathcal U\times \mathcal V\rightarrow \mathcal G\times \mathcal M$ 
be a diffeomorphism onto its image such that 
\begin{enumerate}[(i)]
 \item $\psi_*(\mathcal D_\mathcal G|_W)=\mathcal D|_{\tilde{\psi}(W)}$ for a 
distribution $\mathcal D$ associated to an infinitesimal action and any open 
subset
$W\subseteq U\times V$, and
\item $\pi_\mathcal G|_{U\times V}=\pi_\mathcal G\circ\psi$.
\end{enumerate}
Then given any element $g\in U$ there exists a diffeomorphism onto its image 
$\rho:\mathcal V\rightarrow \mathcal M$ 
such that $\psi\circ\iota_{g}|_V=\iota_{g}|_{\tilde{\rho}(V)}\circ \rho$. Hence 
$\psi$ is a flat chart with respect  to
$(\mathcal D, U, V, g,\rho)$.
\end{lemma}

\begin{proof}
The property $\pi_\mathcal G \circ \psi=\pi_\mathcal G|_{U\times V}$ implies 
$(\tilde{\psi}\circ \tilde{\iota}_{g})(V)\subset \{g\}\times M$.
To show that $\psi\circ\iota_{g}|_V:\mathcal V\rightarrow \mathcal 
G\times\mathcal M$ induces a map $\rho:\mathcal V\rightarrow \mathcal 
M\cong\{g\}\times\mathcal M$,
it is enough to check that $(\psi\circ\iota_{g}|_V)^*\circ\pi_\mathcal 
G^*=\mathrm{ev}_g$, where $\mathrm{ev}_g:\mathcal O_\mathcal G(G)\rightarrow \R$ 
denotes the evaluation in $g$. 

The fact $\pi_\mathcal G\circ \psi=\pi_\mathcal G|_{U\times V}$ implies
$(\psi\circ\iota_g|_V)^*\circ\pi_\mathcal G^*
=(\pi_\mathcal G\circ\psi\circ\iota_g|_V)^*=(\pi_\mathcal G|_{U\times V}\circ 
\iota_g|_V)^*
=\mathrm{ev}_g$ since $(\pi_\mathcal G\circ \iota_g)$ is the unique map 
$\mathcal M\rightarrow \{g\}\subset \mathcal G$.

The map $\rho$ satisfies 
$\psi\circ\iota_{g}|_V=\iota_{g}|_{\tilde{\rho}(V)}\circ \rho$ by definition and 
is a diffeomorphism onto its image since $\psi$ is 
a diffeomorphism onto its image
and $\mathcal V$ and $\mathcal M$ have the same dimension.
\end{proof}

\begin{prop}[Uniqueness of flat charts] \label{prop: uniqueness of flat charts}
If $\mathcal D$ is a distribution on $\mathcal G\times \mathcal M$ associated to 
an infinitesimal action $\lambda:\g\rightarrow \mathrm{Vec}(\mathcal M)$, then a
 flat chart $\psi:\mathcal U\times \mathcal V\rightarrow \mathcal G\times 
\mathcal M$ with respect to $(\mathcal D, U, V, g, \rho)$ is unique.
\end{prop}
 The proof of this proposition is carried out in two steps:
\begin{enumerate}[(i)]
\item First, it is shown that two flat charts $\psi_1$ and $\psi_2$ with respect 
to $(\mathcal D, U, V, g, \rho)$,
coincide on an open neighbourhood of $\{g\}\times V$ in $U\times V$.
\item
Second, the local statement and the connectedness of $U$ are used to globally 
get $\psi_1=\psi_2$.
\end{enumerate}

\begin{proof} (i)
Consider first the case where $\rho=\mathrm{id}$ and $\mathcal D=\mathcal 
D_\mathcal G$, i.e. the infinitesimal action $\lambda$ is the zero map. Now, 
let 
$\psi$ be a flat chart with respect to $(\mathcal D_\mathcal 
G,U,V,g,\mathrm{id})$. Since the identity  $\mathrm{id}_{\mathcal G\times 
\mathcal M}|_{U\times V}$ is
also a flat chart with respect to $(\mathcal D_\mathcal G, U,V, g,\mathrm{id})$, 
the equality of $\psi$ and the identity in a neighbourhood of any point $(g,p)$ 
for $p\in V$
 needs to be shown.\\
Let $(t,\tau)$ be local coordinates on a connected neighbourhood $U'$ of $g\in 
G$ such that $g=0$ in the coordinates
 and let $(x,\theta)$ be local coordinates on a neighbourhood $V'$
of $p\in V$. Then $(t,\tau, x,\theta)$ are local coordinates for $\mathcal 
G\times \mathcal M$ in a neighbourhood of $(0,p)=\tilde{\psi}(0,p)$ 
and therefore on 
a neighbourhood of $\tilde{\psi}(U''\times V'')\subseteq U'\times V'$ for 
appropiate subsets $U''\subseteq U'$ and $V''\subseteq V'$.
Since $\pi_\mathcal G|_{U\times V}=\pi_\mathcal G\circ \psi$, we have 
$$\psi^*(t,\tau)=\psi^*(\pi_\mathcal G^*(t,\tau))=\pi_\mathcal 
G^*(t,\tau)=(t,\tau).$$
Moreover, $\psi_*(\mathcal D_\mathcal G)=\mathcal D_\mathcal G$ implies 
$\psi_*\left(\frac{\partial}{\partial t_r}\right),\, 
\psi_*\left(\frac{\partial}{\partial \tau_s}\right)
\in\mathrm{span} \left\{ \frac{\partial}{\partial 
t_1},\ldots,\frac{\partial}{\partial t_k},\frac{\partial}{\partial 
\tau_1},\ldots,\frac{\partial}{\partial \tau_l}\right\}$
for all $r$, $s$. 
Hence $$\psi_*\left(\frac{\partial}{\partial 
t_r}\right)(x,\theta)=\psi_*\left(\frac{\partial}{\partial 
\tau_s}\right)(x,\theta)=(0,0),$$ and then  
$\frac{\partial}{\partial 
t_r}(\psi^*(x,\theta))=\psi^*(\psi_*(\frac{\partial}{\partial 
t_r})(x,\theta))=\psi^*(0,0)=(0,0)$ and similarly 
$\frac{\partial}{\partial \tau_s}(\psi^*(x,\theta))=(0,0)$. 
Consequently, we have 
$$\psi^*(x,\theta)=\iota_0^*\circ 
\psi^*(x,\theta)=\iota_0^*(x,\theta)=(x,\theta)$$
and thus $\psi=\mathrm{id}$ on $U''\times V''$.\\
Let $\psi_1$ and $\psi_2$ now be flat charts with respect to $(\mathcal D, 
U,V,g,\rho)$ for a distribution $\mathcal D$ associated to an arbitrary 
infinitesimal action and
arbitrary~$\rho$. For any $p\in V$, we have 
$$\tilde{\psi}_1(g,p)=\tilde{\psi}_1\circ\tilde{\iota}_g(p)=\tilde{\iota}
_g(\tilde{\rho}(p))=(g,\tilde{\rho}(p))=\tilde{\psi}_2\circ\tilde{\iota}
_g(p)=\tilde{\psi}_2(g,p).$$ 
Now let $U'$ and $U''$ be open connected neighbourhoods 
of $g$ and $V'$ and $V''$ open neighbourhoods of $p$ with $U''\subseteq U'$ and 
$V''\subseteq V'$ such that the associated open subsupermanifolds $\mathcal U'$ 
of
$\mathcal G$ and $\mathcal V'$ of $\mathcal M$ are isomorphic to superdomains 
and 
$\tilde{\psi}_1(U''\times V'')$ is contained in $\tilde{\psi}_2(U'\times V')$. 
The composition $\psi_2^{-1}\circ \psi_1$ is defined on $U''\times V''$ and a 
flat chart with
respect to $(\mathcal D_\mathcal G, U'',V'',g,\mathrm{id})$. By the above 
argumentation, $\psi_2^{-1}\circ \psi_1=\mathrm{id}$ and thus $\psi_1=\psi_2$  
on $U''\times V''$.\vspace{5.5pt}\\
\noindent(ii)
Let $\psi_1$ and $\psi_2$ be two flat charts with respect to $(\mathcal D, U, 
V,g, \rho)$. For each $p\in V$ define the subset $W_p\subseteq U$ containing the 
points $t\in U$ such that $\psi_1=\psi_2$ on an open neighbourhood of $(t,p)$ in 
$U\times V$.
The sets $W_p$ are open by definition and contain $g$ as a consequence of (i). To 
prove $\psi_1=\psi_2$, i.e. $W_p=U$ for each $p$, it is therefore enough to show 
that each 
$W_p$ is also closed in $U$ due to the connectedness of $U$.\\
If $W_p$ is not closed, then there is a point $t_0\in U\setminus W_p$ such that 
$W_p\cap \Omega\neq\emptyset$ for every open neighbourhood $\Omega$ of $t_0$. 
The continuity of the underlying maps implies 
$\tilde{\psi}_1(t_0,p)=\tilde{\psi}_2(t_0,p)$.
Let $U'$ be an open neighbourhood of $t_0$ in $U$ and $V'$ and open 
neighbourhood of $p$ in $V$ such that the associated open subsupermanifolds 
$\mathcal U'$ and 
$\mathcal V'$ are isomorphic to superdomains. Now let $U''\subseteq U'$ be an 
open connected neighbourhood of $t_0$ and $V''\subseteq V'$ an open 
neighbourhood of
$p$ such that $\tilde{\psi}_1(U''\times V'')\subseteq \tilde{\psi_2}(U'\times 
V')$. Moreover, let $s_0$ be an element of $U''\cap W_p$, which exists by 
assumption on the choice of $t_0$.
After shrinking $V'$ and $V''$, the maps $\psi_1$ and $\psi_2$ coincide on an 
open neigbhourhood of $\{s_0\}\times V'$. By Lemma \ref{lemma: almost flat 
chart} there exists a 
diffeomorphism onto its image $\rho_0:\mathcal V'\rightarrow \mathcal M$ such 
that the restrictions of $\psi_1$ and $\psi_2$ to $U'\times V'$ are flat charts 
with respect to $(\mathcal D, U',V',s_0,\rho_0)$. The same argument as given in 
(i) then shows that $\psi_1$ and $\psi_2$ coincide on $U''\times V''$ which is a 
contradiction
to the assumption $t_0\notin W_p$.
\end{proof}

\begin{rmk}\label{rmk: equality of flat charts and even parts}
 Let $\psi$ and $\psi'$ be flat charts and denote by $\psi_0$ and $\psi_0'$ the 
associated flat charts for the even part 
 (cf. Lemma~\ref{lemma: extension of infinitesimal action}).
 Then $\psi=\psi'$ if and only if $\psi_0=\psi_0'$.
 
 Moreover, if $\psi_0:U\times \mathcal V\rightarrow G\times \mathcal M$ is a 
flat chart with respect to 
 $(\mathcal D_0,U, V, g,\rho)$ then the uniqueness and local existence of flat 
charts imply 
 that there is a flat chart $\psi:\mathcal U\times\mathcal V\rightarrow\mathcal 
G\times\mathcal M$ 
 with respect to $(\mathcal D, U, V, g,\rho)$.
\end{rmk}

\subsection{Uniqueness of local actions}

\begin{lemma}\label{lemma: local action and flat charts}
 Let $\varphi:\mathcal W\subseteq \mathcal G\times\mathcal M\rightarrow \mathcal 
M$ be the local action of the Lie supergroup $\mathcal G$
on a supermanifold $\mathcal M$ with induced infinitesimal action 
$\lambda_\varphi$. Let $\mathcal D$ denote the distribution associated to 
$\lambda_\varphi$.
Define $$\psi=(\mathrm{id}_\mathcal G\times \varphi)\circ 
(\mathrm{diag}\times\mathrm{id}_\mathcal M):\mathcal W \rightarrow \mathcal 
G\times \mathcal M$$ where 
$\mathrm{diag}:\mathcal G\rightarrow \mathcal G\times \mathcal G$ denotes the 
diagonal.

Let $p\in M$ and $U\subset W_p=\{g\in G|\,(g,p)\in W\}$ be a relatively compact 
connected open neighbourhood of $e\in W_p\subset G$. 
Then there exists an open neighbourhood $V$ of $p\in M$ with $U\times V\subset 
W$. 
The restriction $\psi|_{U\times V}$ is a flat chart with respect to $(\mathcal 
D, U, V, e, \mathrm{id})$.
\end{lemma}

\begin{proof}
Since $\overline{U}\times \{p\}\subset W$ is compact, we can find an open 
neighbourhood $V$ of $p$ with 
$U\times V\subseteq W$.
By definition of $\psi$, we have $\pi_\mathcal G\circ \psi=\pi_\mathcal G$. 
Moreover,
$\psi\circ\iota_e=\iota_e$ since $\varphi$ is a local action.

Let $X\in \g$, then 
$X\circ \mathrm{diag}^*=\mathrm{diag}^*\circ (X\otimes\mathrm{id}_\mathcal G^* 
+\mathrm{id}_\mathcal G^*\otimes X)$ since $X$ is a derivation. 
If $\iota_e^\mathcal G:\mathcal G\hookrightarrow \{e\}\times \mathcal G\subset 
\mathcal G\times \mathcal G$ is the inclusion, then 
$\mu\circ\iota_e^\mathcal G=\mathrm{id}_\mathcal G$ for the multiplication $\mu$ 
on $\mathcal G$. 
Since $X$ is right-invariant, we have
$X=(X(e)\otimes \mathrm{id}_\mathcal G^*)\circ \mu^*.$
By a calculation, using these facts and that $\varphi$ is a local action, we 
obtain
$$ \psi_*(X\otimes \id)=X\otimes \id+\id\otimes \lambda_{\varphi}(X),$$
which yields $\psi_*(\mathcal D_\mathcal G)=\mathcal D$. 
\end{proof}

The lemma implies that every local action of a Lie supergroup is uniquely 
determined by its domain of definition and its induced infinitesimal action:

\begin{cor}\label{cor: uniqueness of local action}
The domain of definition and the induced infinitesimal action uniquely determine 
a local action, i.e.
 if $\varphi_1:\mathcal W\rightarrow \mathcal G\times \mathcal M$ and 
$\varphi_2: \mathcal W\rightarrow \mathcal G\times \mathcal M$ are two local 
actions of 
the Lie supergroup $\mathcal G$ on the supermanifold $\mathcal M$ with the same 
induced infinitesimal action $\lambda=\lambda_{\varphi_1}=\lambda_{\varphi_2}$, 
then
$\varphi_1=\varphi_2$.
\end{cor}
\begin{proof}
 By the preceding lemma and the uniqueness of flat charts we have 
$\psi_1=(\mathrm{id}_\mathcal G\times \varphi_1)\circ 
(\mathrm{diag}\times\mathrm{id}_\mathcal M)$
and $\psi_2=(\mathrm{id}_\mathcal G\times \varphi_2)\circ 
(\mathrm{diag}\times\mathrm{id}_\mathcal M)$ and hence $\varphi_1=\pi_\mathcal 
M\circ \psi_1=\pi_\mathcal M\circ \psi_2=\varphi_2$.
\end{proof}

\subsection{Construction of a local action}
In the following, let $\lambda:\g\rightarrow \mathrm{Vec}(\mathcal M)$ be a 
fixed infinitesimal action and $\mathcal G$ a Lie supergroup with multiplication 
$\mu:\mathcal G\times\mathcal G\rightarrow\mathcal G$ and
Lie superalgebra of right-invariant vector fields $\g$. 

The goal is now to find a local $\mathcal G$-action on $\mathcal M$ with induced 
infinitesimal 
action $\lambda$. 
Such a local action of $\mathcal G$ on $\mathcal M$ is constructed using flat 
charts for the distribution $\mathcal D$ associated to $\lambda$.
The domain of definition of the constructed action depends, in general, on some 
choices. After restricting two local actions with a the same infinitesimal 
action a neighbourhood of $\{e\}\times  M$ in $G\times M$ the local actions 
coincide as proven in the previous paragraph. 
Nevertheless, in general there is, as in the classical case (cf. \cite{Palais} 
Chapter III.4),
no unique maximal domain of definition on which the local action can be 
defined.\vspace{\baselineskip}

\noindent\textbf{Definition of a local action:}\\
Choose a neighbourhood basis $\{U_\alpha\}_{\alpha\in A}$ of the identity $e\in 
G$ such that (cf. \cite{Palais}, Chapter II, \S 7):
\begin{enumerate}[(i)]
 \item Each $U_\alpha$ is connected and $U_\alpha=(U_\alpha)^{-1}=\{g\in 
G|\,g^{-1}\in U_\alpha\}$.
 \item For $\alpha,\beta \in A$ either $U_\alpha\subseteq U_\beta$ or 
$U_\beta\subseteq U_\alpha$ holds.
\end{enumerate}
Note that the two conditions guarantee the connectedness of $U_\alpha\cap 
U_\beta$ for arbitrary $\alpha,\beta\in A$.

For each $p\in M$ choose $\alpha(p)\in A$ and a neighbourhood $V_p\subseteq M$ 
of $p$ such that there is a flat chart 
$$\psi_p:\mathcal U_{\alpha(p)}^2\times \mathcal V_p\rightarrow \mathcal G\times 
\mathcal M$$ with respect to $(\mathcal D,U_{\alpha(p)}^2,V_p,e,\mathrm{id})$, 
where $\mathcal U_{\alpha(a)}^2$ and $\mathcal V_p$ denote again the open 
subsupermanifolds of $\mathcal G$ and $\mathcal M$ with underlying sets 
$U_{\alpha(p)}^2=\{gh|\,g,h\in U_\alpha\}$ and $V_p$.
For two elements $p,q\in M$, we may assume  $U_{\alpha(p)}\subseteq 
U_{\alpha(q)}$. Therefore, if the intersection 
$$(U_{\alpha(p)}^2\times V_p )\cap (U_{\alpha(q)}^2\times 
V_q)=U_{\alpha(p)}^2\times (V_p\cap V_q)$$ is non-empty, then the restrictions 
of $\psi_p$ and $\psi_q$ to their common
domain of definition are both flat charts with respect to $(\mathcal D, 
U_{\alpha(p)}^2, (V_p\cap V_q), e,\mathrm{id})$ and hence coincide.
Let 
$$W=\bigcup_{p\in M} (U_{\alpha(p)}\times V_p)\subseteq G\times M.$$
The set $W$ is open by definition and contains $\{e\}\times M$. Furthermore, for 
each $p\in M$ the subset $W_p=\{g\in G|\, (g,p)\in W\}\subseteq G$ is connected 
since all
$U_{\alpha(q)}$ are connected. Let $ \mathcal W=(W,\mathcal O_{G\times \mathcal 
M}|_W)$ and  define a morphism 
$\psi:\mathcal W\rightarrow G\times \mathcal M$
by $\psi|_{U_{\alpha(p)}\times V_p}=\psi_p$ for each $p\in M$. Let
$$\varphi:\mathcal W\rightarrow \mathcal M,\, \varphi=\pi_\mathcal M\circ\psi.$$
We now show that $\varphi $ defines a local group action with induced 
infinitesimal action~$\lambda$.

\begin{prop}\label{prop: local action}
 The map $\varphi$ defines a local action of $\mathcal G$ of the supermanifold 
$\mathcal M$.
\end{prop}
\begin{proof}
Since the map $\psi$ defining $\varphi=\pi_\mathcal M\circ \psi$ is locally 
given by flat charts with respect to $(\mathcal D, U_{\alpha(p)}, 
V_p,e,\mathrm{id})$, we have 
$\psi\circ \iota_e=\iota_e$ and therefore
$\varphi\circ\iota_e=\pi_\mathcal M\circ\psi\circ\iota_e=\pi_\mathcal 
M\circ\iota_e=\mathrm{id}_\mathcal M.$
Thus it remains to show that 
\begin{equation}
\varphi\circ (\mu\times \mathrm{id}_\mathcal M)=\varphi\circ 
(\mathrm{id}_\mathcal G\times \varphi)\tag{$\star$}\end{equation}
on the open subsupermanifold of $\mathcal G\times \mathcal G\times\mathcal M$ 
where both sides are defined.
To prove ($\star$) the special form of the distribution associated to an 
infinitesimal action with respect to the group structure of $\mathcal G$ is 
used. 
The commutativity of the following diagram will be shown:
$$\xymatrix{
\mathcal G\times \mathcal G\times \mathcal M\ar[rrrr]^{\chi\times 
\mathrm{id}_\mathcal M}\ar[dd]_{\mathrm{id}_\mathcal G\times \psi}   
&&&&\mathcal G\times \mathcal G\times \mathcal M \ar[dd]^{(\chi^{-1}\times 
\mathrm{id}_\mathcal M)\circ (\mathrm{id}_\mathcal G\times \psi)} \\
&&&&\\
\mathcal G\times \mathcal G\times \mathcal M \ar[rrrr]_{(\tau\times 
\mathrm{id}_\mathcal M)\circ (\mathrm{id}_\mathcal G\times \psi)\circ 
(\tau\times \mathrm{id}_\mathcal M)}
&&&& \mathcal G\times \mathcal G\times \mathcal M
}$$

In the above diagram all maps are only defined on appropriate open 
subsupermanifolds of $\mathcal G\times \mathcal G\times \mathcal M$.
The map $\tau:\mathcal G\times \mathcal G\rightarrow \mathcal G\times \mathcal 
G$ denotes the map which interchanges the two components. 
Moreover, the map $\chi:\mathcal G\times \mathcal G\rightarrow \mathcal G\times 
\mathcal G$ is defined by 
$$\chi=(\mathrm{id}_\mathcal G\times \mu)\circ 
(\mathrm{diag}\times\mathrm{id}_\mathcal G),$$  
such that $\mu=\pi_2\circ\chi$ and $\pi_1=\pi_1\circ \chi$, where $\pi_i$, 
$i=1,2$,
is the projection onto the $i$-th factor.
The underlying map is given by $\tilde\chi(g,h)=(g,gh)$.

Note that if $\mathcal G= \mathcal M$ and the infinitesimal action $\lambda$ is 
the canonical inclusion
$\mathfrak g\hookrightarrow \mathrm{Vec}(\mathcal G)$ of the right-invariant 
vector fields, then $\chi$ is a flat chart for the distribution $\mathcal D$
with respect to $(\mathcal D, G, M,e,\mathrm{id})$, $\chi$ and $\psi$ coincide 
(on their common domain of definition) and ($\star$) is
equivalent to the associativity of the multiplication $\mu$.

Let 
$$\Psi_1=(\tau\times \mathrm{id}_\mathcal M)\circ (\mathrm{id}_\mathcal G\times 
\psi)\circ (\tau\times \mathrm{id}_\mathcal M)\circ(\mathrm{id}_\mathcal G\times 
\psi)$$
and 
$$\Psi_2= (\chi^{-1}\times \mathrm{id}_\mathcal M)\circ (\mathrm{id}_\mathcal 
G\times \psi)\circ(\chi\times \mathrm{id}_\mathcal M).$$
The underlying maps are 
$\tilde\Psi_1(g,h,p)=(g,h,\tilde\varphi(g,\tilde\varphi(h,p)))$ 
and $\tilde\Psi_2(g,h,p)=(g,h,\tilde\varphi(gh,p))$.
The open subsupermanifold of $\mathcal G\times \mathcal G\times \mathcal M$ on 
which both morphisms $\Psi_1$ and $\Psi_2$ are defined is exactly the open 
subsupermanifold
which is the common domain of definition of $\varphi\circ (\mu\times 
\mathrm{id}_\mathcal M)$ and $\varphi\circ (\mathrm{id}_\mathcal G\times 
\varphi)$.
A calculation shows
$\pi_\mathcal M\circ \Psi_1
=\varphi\circ (\mathrm{id}_\mathcal G\times \varphi)$
and 
 $\pi_\mathcal M\circ \Psi_2
=\varphi\circ (\mu\times \mathrm{id}_\mathcal M)$
and the commutativity of the above diagram directly implies ($\star$).

To show the equality of $\Psi_1$ and $\Psi_2$, we consider the distribution 
$\mathcal D_{\id\otimes \lambda}$
on $\mathcal G\times (\mathcal G\times \mathcal M)$ associated to the 
infinitesimal action $\id\otimes \lambda:\g\rightarrow \mathrm{Vec}(\mathcal 
G\times \mathcal M)$,
$X\mapsto (\mathrm{id}_\mathcal G^*\otimes \lambda(X))$,
 of $\mathcal G$ on $\mathcal G\times \mathcal M$ and show that $\Psi_1$ and 
$\Psi_2$ are locally both flat charts for this distribution. 

Let $\mathcal D_\mathcal G$ be the distribution on $\mathcal G\times \mathcal 
G\times \mathcal M$ which is spanned by vector fields of the form
$X\otimes \mathrm{id}_\mathcal G^*\otimes \mathrm{id}_\mathcal M^*$ for $X\in 
\g$. The distribution $\mathcal D_{\id\otimes\lambda}$ is spanned by vector 
fields of the form
$X\otimes \mathrm{id}_\mathcal G^*\otimes \mathrm{id}_\mathcal 
M^*+\mathrm{id}_\mathcal G^*\otimes \mathrm{id}_\mathcal G^*\otimes \lambda(X)$.
For $X\in \g$ and writing $\id$ for $\mathrm{id}_\mathcal G^*$ or 
$\mathrm{id}_\mathcal M^*$, we have
\begin{align*}
 (&\Psi_1)_*(X\otimes \id\otimes \id)
=(\tau\times \mathrm{id}_\mathcal M)_*(\mathrm{id}_\mathcal G\times 
\psi)_*(\tau\times \mathrm{id}_\mathcal M)_*(\mathrm{id}_\mathcal G\times 
\psi)_*(X\otimes \id\otimes\id)\\
&=(\tau\times \mathrm{id}_\mathcal M)_*(\mathrm{id}_\mathcal G\times 
\psi)_*(\id\otimes X\otimes \id)
=(\tau\times \mathrm{id}_\mathcal M)_*(\id\otimes X\otimes 
\id+\id\otimes\id\otimes \lambda(X))\\
&=X\otimes \id\otimes \id+\id\otimes \id\otimes \lambda(X),
\end{align*}
where the fact that $\psi$ is a flat chart for the distribution $\mathcal D$ 
associated $\lambda$ is used.
Similarly, using the fact that $\chi_*(X\otimes \id)=(X\otimes \id+\id\otimes 
X)$, we get
$ (\Psi_2)_*(X\otimes \id\otimes \id)=X\otimes \id\otimes \id+\id\otimes 
\id\otimes \lambda(X).$
Hence, $\Psi_1$ and $\Psi_2$ both transform $\mathcal D_\mathcal G$ into 
$\mathcal D_{\id\otimes\lambda}$, i.e.
$(\Psi_i)_*(\mathcal D_\mathcal G)=\mathcal D_{\id\otimes \lambda}$.

Remark that $\pi_\mathcal G\circ\Psi_1=\pi_\mathcal G=\pi_\mathcal G\circ 
\Psi_2$ if $\pi_\mathcal G:\mathcal G\times \mathcal G\times \mathcal 
M\rightarrow \mathcal G$
denotes the projection onto the first component.
Now, let
$\iota_e^{\mathcal G\times \mathcal M}:\mathcal G\times \mathcal 
M\hookrightarrow \{e\}\times \mathcal G\times \mathcal M\subset\mathcal 
G\times\mathcal G\times \mathcal M$
and 
$\iota_e^{\mathcal M}:\mathcal M\hookrightarrow \{e\}\times \mathcal M\subset 
\mathcal G\times \mathcal M$
denote the inclusions. Then a calculation shows 
$$ \Psi_1\circ\iota_e^{\mathcal G\times \mathcal M}
=\iota_e^{\mathcal G\times \mathcal M}\circ \psi
\ \ \ \text{ and }\ \ \ 
 \Psi_2\circ \iota_e^{\mathcal G\times \mathcal M}
=\iota_e^{\mathcal G\times \mathcal M}\circ\psi
$$
since $\psi\circ \iota_e^\mathcal M=\iota_e^\mathcal M$ and $(\chi\times 
\mathrm{id}_\mathcal M)\circ \iota_e^{\mathcal G\times \mathcal M}
=\iota_e^{\mathcal G\times \mathcal M}$ using that $e$ is the identity element 
of $\mathcal G$ and the definition of $\chi$.
This shows that $\Psi_1$ and $\Psi_2$ are both locally flat charts in $e$.

In order to check that $\Psi_1$ and $\Psi_2$ coincide everywhere, the special 
form of the sets in $\{U_\alpha\}_{\alpha\in A}$ is important.
Let $(g,g',p)\in G\times G\times M$ such that $\Psi_1$ and $\Psi_2$ are defined 
on a neighbourhood of $(g,g',p)$, i.e. such that
$(g',p)$, $(gg',p)$, $(g,\tilde{\varphi}(g,p))\in W$.
Then by definition of $\psi$, there exists $U_\alpha\in \{U_\gamma\}_{\gamma\in 
A}$ containing $g$, a neighbourhood $V_\alpha$ of $q=\tilde{\varphi}(g',p)$ in 
$M$
and a flat chart $\psi_\alpha:\mathcal U_\alpha^2\times \mathcal 
V_\alpha\rightarrow \mathcal G\times \mathcal M$ with respect to 
$(\mathcal D,U_\alpha^2, V_\alpha,e, \mathrm{id})$ and with 
$\psi|_{U_\alpha\times V_\alpha}=\psi_\alpha|_{U_\alpha\times V_\alpha}.$
Furthermore, choose $U_\beta\in\{U_\gamma\}_{\gamma\in A}$ containing $g'$ and 
$gg'$ and a neighbourhood $V_\beta$ of $p$ in $M$
such that there exists a flat chart 
$\psi_\beta:\mathcal U_\beta^2\times \mathcal V_\beta\rightarrow \mathcal 
G\times \mathcal M$ with respect to
$(\mathcal D, U_\beta^2,V_\beta, e,\mathrm{id})$ and with
$\psi|_{U_\beta\times V_\beta}=\psi_\beta|_{U_\beta\times V_\beta}.$

Shrink $V_\beta$ and choose a neighbourhood $U$ of $g'$ in $G$ with $U\subseteq 
U_\beta$ such that $\tilde{\varphi}(U\times V_\beta)\subseteq V_\alpha$.
By the special choice of the neighbourhood basis $\{U_\gamma\}_{\gamma\in A}$ of 
$e$ in $G$ either $U_\alpha\subseteq U_\beta$ or $U_\beta \subseteq U_\alpha$ is 
true.

First, let $U_\alpha\subseteq U_\beta$.
The map $\Psi_1$ is defined of $U_\alpha\times U\times V_\beta$ and the 
restriction of $\Psi_1$ to $U_\alpha\times U\times V_\beta$ is a flat chart with 
respect to 
$(\mathcal D_{\id\otimes \lambda},U_\alpha,U\times V_\beta,e,\psi)$.
Moreover, we have $\tilde{\mu}(U_\alpha\times U)=U_\alpha\cdot U\subseteq 
U_\beta^2$ and thus 
$\Psi_{2,\beta}=(\chi^{-1}\times \mathrm{id}_\mathcal M)\circ 
(\mathrm{id}_\mathcal G\times \psi_\beta)\circ(\chi\times \mathrm{id}_\mathcal 
M)$ 
is also defined on $U_\alpha \times U\times V_\beta$ and a flat chart 
with respect to $(\mathcal D_{\id\otimes \lambda},U_\alpha,U\times V_\beta, 
e,\psi)$. 
By the uniqueness of flat charts, $\Psi_1$ and $\Psi_{2,\beta}$ coincide on 
$U_\alpha\times U\times V_\beta$ so that $\Psi_1=\Psi_2$ near $(g,g',p)\in 
U_\alpha\times U\times V_\beta$. 

Consider now the case $U_\beta\subseteq U_\alpha$. The map 
$\Psi_{1,\alpha}=(\tau\times \mathrm{id}_\mathcal M)\circ (\mathrm{id}_\mathcal 
G\times \psi_\alpha)\circ (\tau\times \mathrm{id}_\mathcal 
M)\circ(\mathrm{id}_\mathcal G\times \psi)$
is defined on $U_\alpha^2\times U\times V_\beta$ and is a flat chart with 
respect to 
$(\mathcal D_{\id\otimes \lambda}, U_\alpha^2, U\times V_\beta,e,\psi)$.
The set $U_\beta(g')^{-1}$ contains $e$ and $g$.
Since $U_\beta (g')^{-1}\subset U_\beta \cdot U_\beta^{-1}=U_\beta ^2\subseteq 
U_\alpha^2$, the map $\Psi_{1,\alpha}$ may be restricted to
$U_\beta (g')^{-1}\times U\times V_\beta$ and is then a flat chart with respect 
to $(\mathcal D_{\id\otimes \lambda},U_\beta (g')^{-1}, U\times 
V_\beta,e,\psi)$.
After possibly shrinking $U$, we may assume $\tilde{\mu}(U_\beta (g')^{-1}\times 
U)=U_\beta (g')^{-1}\cdot U\subseteq U_\beta ^2$. 
Then $\Psi_2$ is defined on $U_\beta (g')^{-1}\times U\times V_\beta$ and is a 
flat chart with respect to 
$(\mathcal D_{\id\otimes \lambda}, U_\beta(g')^{-1}, U\times V_\beta,e,\psi)$. 
Again, the uniqueness of flat charts implies $\Psi_1=\Psi_2$ near $(g,g',p)\in 
U_\beta(g')^{-1}\times U\times V_\beta$.
\end{proof}

\begin{prop}\label{prop: infinitesimal action as required}
 The infinitesimal action $\lambda_\varphi:\g\rightarrow \mathrm{Vec}(\mathcal 
M),\,X\mapsto (X(e)\otimes \mathrm{id}_\mathcal M^*)\circ \varphi^*$ induced by
the local $\mathcal G$-action $\varphi$ is $\lambda$.
\end{prop}

\begin{proof}
Since the map $\psi$ is locally a flat chart, $\psi $ is a local 
diffeormorphism, $\pi_\mathcal G\circ \psi=\pi_\mathcal G$ and $\psi\circ 
\iota_e=\iota_e$. 
Moreover, we have locally $\psi_*(\mathcal D_\mathcal G)=\mathcal D$. Therefore 
the vector 
field $\psi_*(X\otimes \mathrm{id}_\mathcal M^*)$ on $\mathcal G\times \mathcal 
M$ belongs to the distribution $\mathcal D$ for each vector field $X\in\g$.
We have 
\begin{align*}
 \psi_*(X\otimes \mathrm{id}_\mathcal M^*)\circ\pi_\mathcal G^*
&=(\psi^{-1})^*\circ (X\otimes \mathrm{id}_\mathcal M^*)\circ \psi^*\circ 
\pi_\mathcal G^*
=(\psi^{-1})^*\circ(X\otimes\mathrm{id}_\mathcal M^*)\circ \pi_\mathcal G^*\\
&=(\psi^{-1})^*\circ \pi_\mathcal G^*\circ X
=\pi_\mathcal G^*\circ X
=(X\otimes \mathrm{id}_\mathcal M^*)\circ\pi_\mathcal G^*.
\end{align*}
 Let $X_1,\ldots,X_{k+l}$ be a basis of $\g$ and $a_1,\ldots,a_{k+l}$ local 
functions on $\mathcal G\times \mathcal M$ such that 
$$\psi_*(X\otimes \mathrm{id}_\mathcal M^*)=\sum_{i=1}^{k+l} a_i(X_i\otimes 
\mathrm{id}_\mathcal M^*+\mathrm{id}_\mathcal G^*\otimes \lambda(X_i)),$$
which is possible since $\psi_*(X\otimes \mathrm{id}_\mathcal M^*)$ belongs to 
$\mathcal D$.
Then, combining the above, we have
\begin{align*}
 (X\otimes \mathrm{id}_\mathcal M^*)\circ\pi_\mathcal G^*
=\left(\sum_{i=1}^{k+l} a_i(X_i\otimes \mathrm{id}_\mathcal 
M^*+\mathrm{id}_\mathcal G^*\otimes \lambda(X_i))\right) \circ\pi_\mathcal G^*
=\left(\left(\sum_{i=1}^{k+l} a_iX_i\right)\otimes \mathrm{id}_\mathcal 
M^*\right) \circ \pi_\mathcal G^*,
\end{align*}
which implies $X=\sum_{i=1}^{k+l} a_i X_i$. Since $X\in \g$ and 
$X_1,\ldots,X_{k+l}$ is a basis of $\g$, the $a_i$'s are all constants and hence
 $$\psi_*(X\otimes \mathrm{id}_\mathcal M^*)=X\otimes \mathrm{id}_\mathcal 
M^*+\mathrm{id}_\mathcal G^*\otimes \lambda(X).$$
Therefore,
\begin{align*}
\lambda_\varphi(X) 
&=(X(e)\otimes \mathrm{id}_\mathcal M^*)\circ\varphi^*
=\iota_e^*\circ (X\otimes \mathrm{id}_\mathcal M^*)\circ\psi^*\circ\pi_\mathcal 
M^*=\iota_e^*\circ \psi^*\circ \psi_*(X\otimes \mathrm{id}_\mathcal M^*)\circ 
\pi_\mathcal M^*\\
&=(\psi\circ\iota_e)^*\circ (X\otimes \mathrm{id}_\mathcal 
M^*+\mathrm{id}_\mathcal G^*\otimes \lambda(X))\circ \pi_\mathcal M^*
=\iota_e^*\circ (0+\pi_\mathcal M^*\circ \lambda(X))=\lambda(X).
\end{align*}\end{proof}

\section{Globalizations of infinitesimal actions}

After we proved the existence of a local action of a Lie supergroup $\mathcal G$ 
on a supermanifold $\mathcal M$ with a given induced 
infinitesimal action $\lambda:\g\rightarrow \mathrm{Vec}(\mathcal M)$, it is 
natural to ask in which cases this extends 
a global $\mathcal G$-action on $\mathcal M$.

A simple way to obtain examples of a local action which is not global is to 
start with an action on a 
supermanifold $\mathcal M'$. 
This action then induces a local action, which is not global, on every 
non-invariant open subsupermanifold $\mathcal M\subset \mathcal M'$

The aim of this section is to characterize all infinitesimal action which 
``arise'' in the just described way from a global action. 
These infinitesimal actions are called globalizable.

In the classical case (see \cite{Palais}, Chapter III), Palais found necessary 
and sufficient conditions for an infinitesimal action to be globalizable, 
allowing the larger manifold $M'$, a globalization of the infinitesimal action, 
to be a possibly non-Hausdorff manifold.

In this section, similar conditions for the existence of globalizations of 
infinitesimal actions of Lie supergroups on supermanifolds 
are proven, and differences to the classical case are pointed out. It is also 
shown by an example that an infinitesimal action on a supermanifold may not be 
globalizable even if its underlying infinitesimal action is.

In analogy to the classical case (cf. \cite{Palais}, Chapter III, Definition 
II), we define the notion of a globalization of an infinitesimal action.

\begin{defi}
 A globalization of an infinitesimal action $\lambda:\g\rightarrow 
\mathrm{Vec}(\mathcal M)$ of a Lie supergroup $\mathcal G$ on a supermanifold 
 $\mathcal M$ is a  pair $(\mathcal M',\varphi')$ with the following properties:
 \begin{enumerate}[(i)]
  \item $\mathcal M'$ is a supermanifold, whose underlying manifold $M'$ is 
allowed to be a non-Hausdorff manifold, and $\mathcal M$ is 
  an open subsupermanifold of $\mathcal M'$, and
  \item $\varphi':\mathcal G\times \mathcal M'\rightarrow \mathcal M'$ is an 
action of the Lie supergroup $\mathcal G$ on $\mathcal M'$ 
   such that its infinitesimal action restricted to $\mathcal M$ coincides with 
$\lambda$, and
  \item $\tilde\varphi'(G\times M)=M'$.
 \end{enumerate}
 If there is no chance of confusion the supermanifold $\mathcal M'$ is also 
called a globalization.
 The infinitesimal action $\lambda$ is called globalizable if there exists a 
globalization $(\mathcal M',\varphi')$
 of $\lambda$.
\end{defi}

\begin{defi}
 The domain of definition $\mathcal W=(W,\mathcal O_\mathcal W)$ of a local 
$\mathcal G$-action $\varphi:\mathcal W\rightarrow \mathcal M$ is called 
maximally balanced if
 $W_p=W_{\tilde \varphi(h,p)}h=\{gh|\,g\in W_{\tilde\varphi (h,p)}\}$ for all 
$(h,p)\in W$, where again
 $W_q=\{g\in G|\,(g,q)\in W\}$ for any $q\in M$.
\end{defi}

\begin{rmk}
 In \cite{Palais} a maximally balanced domain of definition of local action is 
called maximum (see Definition VII in Chapter III, \cite{Palais}). 
 Here, the term maximally balanced is used instead in order to avoid any 
confusion with maximal domains of definition.
\end{rmk}

In the classical case, the following theorem states necessary and sufficient 
conditions for the existence of a globalization of an infinitesimal 
action.

\begin{thm}[see \cite{Palais}, Chapter III, Theorem X]\label{thm: classical 
globalizations}
Let $\lambda:\g\rightarrow \mathrm{Vec}(M)$ be an infinitesimal action of a Lie 
group $G$ on a manifold $M$.
Then the following statements are equivalent:
\begin{enumerate}[(i)]
 \item The infinitesimal action $\lambda$ is globalizable.
 \item There exists a local action $\varphi:W\rightarrow M$ of $G$ on $M$ with 
infinitesimal action $\lambda$ whose domain of definition $W$ 
 is maximally balanced; this local action with a maximally balanced domain of 
definition is then unique and any other local action 
 with the same infinitesimal action $\lambda$ is a restriction of $\varphi$.
 \item Let $\mathcal D$ be the distribution on $G\times M$ associated to the 
infinitesimal action $\lambda$ and $\Sigma\subset G\times M$
 any leaf of $\mathcal D$, i.e. a maximal connected integral manifold. Then the 
map $\pi_G|_\Sigma:\Sigma\rightarrow M$ is injective,
 where $\pi_G:G\times M\rightarrow G$ denotes the projection onto $G$.
\end{enumerate}
\end{thm}

\begin{rmk}
 If $\varphi':\mathcal G\times\mathcal M'\rightarrow \mathcal M'$ is a 
globalization of the infinitesimal action $\lambda$,
 then the underlying action $\tilde \varphi ':G\times M'\rightarrow M'$ is a 
globalization of the reduced infinitesimal action
 $\tilde \lambda:\g_0\rightarrow \mathrm{Vec}(M)$. Therefore, a necessary 
condition for an infinitesimal action to be globalizable is the 
 existence of a globalization $M'$ of the reduced infinitesimal action.
\end{rmk}

\subsection{Univalent leaves and holonomy}
Throughout the rest of this section, let $\lambda:\g\rightarrow 
\mathrm{Vec}(\mathcal M)$ be a fixed infinitesimal action of a Lie supergroup 
$\mathcal G$ and let $\mathcal D$ be the associated distribution on $\mathcal 
G\times \mathcal M$. 
In the following, conditions for the existence of a globalization of $\lambda$ 
are studied.

Trying to generalize the classical result (Theorem \ref{thm: classical 
globalizations}) to supermanifolds, the question of 
an appropriate formulation of condition $(iii)$ arises since integral manifolds 
do not uniquely determine 
a distribution on a supermanifold (cf. Example \ref{ex: non-involutive 
distribution}).
For that purpose the notion of a univalent leaf $\Sigma\subset G\times M$, 
extending the classical notion,
is introduced. 

Then occurring holonomy phenomena of the distribution $\mathcal D$ are studied 
and a connection between absence of such phenomena and 
the notion of univalent leaves is established.

\begin{rmk}
 The infinitesimal action $\lambda:\g\rightarrow \mathrm{Vec}(\mathcal M)$ 
induces an infinitesimal action $\tilde{\lambda}:\g_0\rightarrow 
\mathrm{Vec}(M)$
 of the classical Lie group $G$ on the manifold $M$, where the Lie algebra of 
$G$ is identified with the even part $\g_0$ of $\g$. For any vector field $X\in 
\g_0(\subseteq \g)$ we define $\tilde{\lambda}(X)$ to be the reduced vector 
field $\tilde{Y}$ 
 of $Y=\lambda(X)$.
\end{rmk}

\begin{lemma}\label{lemma: underlying distribution}
 Let $\tilde{\lambda}:\g_0\rightarrow \mathrm{Vec}(M)$ denote the induced 
infinitesimal action of $G$ on $M$ and 
 $\mathcal D_{\tilde{\lambda}}$ the associated distribution on $G\times M$. 
 Then we have $\widetilde{\mathcal D}=\mathcal D_{\tilde{\lambda}}$, where 
$\widetilde{\mathcal D}$ denotes the 
 distribution on $G\times M$ induced by $\mathcal D$ (cf. Remark \ref{rmk: 
reduced distribution}).
\end{lemma}

\begin{proof}
 For any odd vector field $Y$ on a supermanifold we have $\tilde Y=0$. 
 Since $X+\lambda(X)$ has the same parity as $X$ if $X$ is homogeneous, the 
reduced distribution
 $\widetilde{\mathcal D}$ is spanned by vector fields of the form $\tilde 
X+\tilde\lambda(X)$ for $X\in \g_0$. These vector fields
 also generate $\mathcal D_{\tilde{\lambda}}$ and thus $\widetilde{\mathcal 
D}=\mathcal D_{\tilde{\lambda}}$.
\end{proof}

As a corollary of the identity $\widetilde{\mathcal D}=\mathcal 
D_{\tilde{\lambda}}$, 
we get the following relation between integral manifolds of 
the involutive distributions $\mathcal D$ on $\mathcal G\times \mathcal M$ and 
$\mathcal D_{\tilde \lambda}$ on $G\times M$.

\begin{cor}
 Every integral manifold $N\subset G\times M$ of the distribution $\mathcal 
D_{\tilde \lambda}$ is the underlying manifold of some
 integral manifold $\mathcal N\subset \mathcal G\times \mathcal M$ of the 
distribution~$\mathcal D$ and conversely
 the underlying manifold of every integral manifold of $\mathcal D$
 is an integral manifold of $\mathcal D_{\tilde \lambda}$.
\end{cor}

In the following, by a leaf $\Sigma\subset G\times M$ a leaf, i.e. a maximal 
connected integral manifold, of the distribution 
$\mathcal D_{\tilde \lambda}=\widetilde {\mathcal D} $ is meant. By the 
preceding corollary every leaf $\Sigma$ is as well the
underlying manifold of an integral manifold of the distribution $\mathcal D$.

The involutiveness of the distribution $\mathcal D_{\tilde \lambda}$ guarantees 
the existence of a leaf through each point $(g,p)\in G\times M$.
This leaf is denoted by $\Sigma_{(g,p)}$.

\begin{defi}\label{defi: univalence}
 A leaf $\Sigma\subset G\times M$ is called univalent (with respect to $\mathcal 
D$) 
 if for every path $\gamma:[0,1]\rightarrow \Sigma\subset G\times M$
 there exists a flat chart 
 $\psi:\mathcal U\times\mathcal V\rightarrow \mathcal G\times \mathcal M$ with 
respect to $(\mathcal D,U,V,\gamma_G(0),\mathrm{id})$, 
 for $\gamma_G\coloneqq \pi_G\circ \gamma:[0,1]\rightarrow G$,
 such that $\tilde \psi(U\times V)$ contains $\gamma([0,1])$.
 
 The infinitesimal action $\lambda$ is called univalent if all leaves 
$\Sigma\subset G\times M$ are univalent.
\end{defi}

\begin{rmk} By Remark~\ref{rmk: flat chart}, a leaf $\Sigma$ is univalent if for 
every path $\gamma:[0,1]\rightarrow\Sigma$ there exists a flat chart $\psi$ with 
respect to $(\mathcal D, U, V, \gamma_G(0),\rho)$ for some $\rho$ or 
equivalently, after shrinking, for any $\rho$, 
with $\gamma([0,1])\subset \tilde\psi(U\times V)$.
\end{rmk}

\begin{rmk}
 In the definition of a univalent leaf $\Sigma$ it is enough for the defining 
property to hold true for closed paths 
 $\gamma:[0,1]\rightarrow \Sigma$ because for any path $\gamma'$ the composition 
$(\gamma')^{-1}\cdot \gamma'$ of $\gamma'$
 and $(\gamma')^{-1}$, $(\gamma')^{-1}(t)=\gamma'(1-t)$, is a closed path with 
$\gamma'([0,1])=((\gamma')^{-1}\cdot \gamma')([0,1])$.
\end{rmk}

\begin{rmk}
 If $\lambda$ is univalent, then so is the induced infinitesimal action 
 $\tilde\lambda:\g_0\rightarrow\mathrm{Vec}(M)$.
\end{rmk}

In \cite{Palais}, an infinitesimal action (in the classical case) is called 
univalent
if the restriction of the projection $\pi_G:G\times M\rightarrow G$ to an 
arbitrary leaf $\Sigma \subset G\times M$ is injective. 
The above defined notion of univalent infinitesimal actions on supermanifolds 
extends this definition:

\begin{prop}
 In the case of classical manifolds $\mathcal G=G$ and $\mathcal M=M$, 
 an infinitesimal action is univalent if and only if
  the projection $\pi_G|_{\Sigma}:\Sigma\rightarrow G$ is injective
  for each leaf $\Sigma\subset G\times M$.
\end{prop}

\begin{proof}
 Let $\Sigma\subset G\times M$ be any leaf, $x=(g,p),y=(g,q)\in \Sigma$, 
 and $\gamma:[0,1]\rightarrow \Sigma$ a path from $x$ to $y$.
 Since $\lambda$ is univalent, there is a flat chart $\psi:U\times V\rightarrow 
G\times M$
 with respect to $(\mathcal D, U, V, g,\mathrm{id})$ with $\gamma([0,1])\subset 
\psi(U\times \{p\})$.
 We have $(g,q)=\gamma(1)=\psi(g,p)=(g,p)$ because $\pi_G\circ \psi=\pi_G$.

 Assume now 
 that $\pi_G|_\Sigma$ is injective for each leaf $\Sigma\subset G\times M$.
 Let $\Sigma \subset G\times M$ be a leaf and $\gamma:[0,1]\rightarrow \Sigma$ a 
path. Using the compactness of 
 $\gamma([0,1])$ there are $0=t_0<\ldots<t_k=1$ and  flat 
 charts $\psi_i:U_i\times V_i\rightarrow G\times M$, $i=0,\ldots,k-1$, with 
respect to $(\mathcal D, U_i, V_i, \gamma_G(t_i),\mathrm{id})$
 such that the intersection $U_i\cap U_j$ is connected for all $i,j$
 and $\gamma([t_{i},t_{i+1}])\subset \psi(U_i\times V_i)$. 
 
 Set $\psi_0'=\psi_0$ and, after possibly shrinking $V_0$, inductively define 
flat charts $\psi_i':U_i\times V_0\rightarrow G\times M$ 
 for $i\geq 1$ by composing $\psi_i$ and a local diffeomorphism of the form 
$(\mathrm{id}\times \rho_i)$ 
 such that $\psi_i'$ and $\psi_{i+1}'$ coincide
 on $(U_i\cap U_{i+1})\times V_0$ for $i=0,\ldots,k-2$ and $\gamma([t_i, 
t_{i+1}])\subset \psi_i'(U_i\times V_0)$ holds. 
 We have $\psi_0'(U_0\times \{p\})\subset \Sigma_{(g,p)}$ for $g\coloneqq 
\gamma_G(0)$ and any $p\in V_0$, and then by induction 
 $\psi_i'(U_i\times \{p\})\subset \Sigma_{(g,p)}$ for any $i$.
 Therefore, $\psi_i'=\psi_j'$ on $(U_i\cap U_j)\times V_0$ since 
$\pi_G|_{\Sigma_{(g,p)}}$ is injective and $\pi_G\circ \psi=\pi_G$ for any flat 
chart $\psi$.
 Consequently, we can define a flat chart $\psi:U\times V\rightarrow G\times M$, 
$U\coloneqq \bigcup_{i=0}^{k-1} U_i$, $V\coloneqq V_0$,
 with respect to $(\mathcal D, U, V, g=\gamma_G(0),\mathrm{id})$ by setting 
$\psi|_{U_i\times V_0}=\psi_i'$. 
 The map $\psi$ satisfies $\gamma([0,1])\subset \psi(U\times V)$ by 
construction.
\end{proof}

\begin{prop}\label{prop: univalence and uniqueness of flat charts}
 The infinitesimal action $\lambda$ is univalent if and only if for any two flat 
charts 
 $\psi_i:\mathcal U_i\times\mathcal V_i\rightarrow \mathcal G\times\mathcal M$ 
with respect to
 $(\mathcal D, U_i, V_i,g,\rho_i)$, $i=1,2$, and $\rho_1=\rho_2$ on $V_1\cap 
V_2$ we have
 $\psi_1=\psi_2$ on their common domain of definition.
\end{prop}
Remark that by Proposition~\ref{prop: uniqueness of flat charts} any two flat 
charts coincide on their common domain of 
definition $(U_1\cap U_2)\times (V_1\cap V_2)$ if $U_1\cap U_2$ is connected.

\begin{proof}
 Let $\lambda$ be univalent and let $\psi_i$ be flat charts with respect to 
 $(\mathcal D, U_i, V_i, g,\rho_i)$ and $\rho_1=\rho_2$ on $V_1\cap V_2$.
 Let $(h,p)\in (U_1\cap U_2)\times (V_1\cap V_2)$.
 Since $\psi_1$ and $\psi_2$ 
 are flat charts, $\tilde \psi_1(U_1\times \{p\})$ and 
 $\tilde \psi_2( U_2\times \{p\})$ are both contained in the leaf 
 $\Sigma=\Sigma_{(g,\tilde\rho_1(p))}=\Sigma_{(g,\tilde\rho_2(p))}$.
 The univalence of the reduced infinitesimal action~$\tilde \lambda$ implies 
 that $\pi_G|_\Sigma:\Sigma\rightarrow G$ is injective and thus 
 $\tilde \psi_1(h,p)=\tilde \psi_2(h,p)$.
 Let $\gamma:[0,1]\rightarrow U_1\cup U_2$ be a closed path with 
$\gamma(0)=\gamma(1)=g$,
 $\gamma([0,\frac 1 2])\subset U_1$, $\gamma([\frac 1 2,1])\subset U_2$
 and $\gamma(\frac 1 2)=h$. Then 
 $$\gamma':[0,1]\rightarrow \Sigma,\, 
 \gamma'(t)=\begin{cases}
             \tilde \psi_1(\gamma(t),p),& t\leq \frac 1 2\\
             \tilde \psi_2(\gamma(t),p),& t> \frac 1 2
            \end{cases} $$
  is a closed path.
 As $\lambda$ is univalent, there is a flat chart $\psi:\mathcal U\times 
\mathcal V\rightarrow\mathcal G\times\mathcal M$
 with respect to $(\mathcal D, U, V, g,\mathrm{id})$ with 
 $\gamma'([0,1])=\tilde\psi_1(\gamma([0,\frac 1 2])\times \{p\})\cup \tilde 
\psi_2(\gamma([\frac 1 2,1])\times \{p\})
 \subset \tilde \psi (U\times V)$.
 Then, after possibly shrinking $V_1$, $\psi\circ (\mathrm{id}\times \rho_1)$ is 
a flat chart with 
 respect to $(\mathcal D, U, V_1, g,\rho_1)$. By Proposition~\ref{prop: 
uniqueness of flat charts} the flat charts
 $\psi\circ (\mathrm{id}\times \rho_1)$ and $\psi_1$ coincide on a neighbourhood 
of $\gamma([0,\frac 1 2])\times \{p\}$. 
 Moreover, $\psi\circ (\mathrm{id}\times \rho_1)$ and $\psi_2$ coincide on a 
neighbourhood of $\gamma([\frac 1 2,1])\times \{p\}$ since
 $\rho_1=\rho_2$ on $V_1\cap V_2$. In particular, we get
 $\psi_1=\psi\circ (\mathrm{id}\times \rho_1)=\psi_2$ near $(h,p)=(\gamma(\frac 
1 2),p)$.
 
 Suppose now that any two flat charts $\psi_i$ with respect to $(\mathcal D, 
U_i,V_i, g,\rho_i)$, $i=1,2$,
 with $\rho_1=\rho_2$ already coicide on their common domain of definition.
 Assume that there is a leaf $\Sigma\subset G\times M$ which is not univalent 
and 
 let $\gamma:[0,1]\rightarrow \Sigma$ be a path for which there is no flat 
chart 
 $\psi:\mathcal U\times \mathcal V\rightarrow \mathcal G\times\mathcal M$ with 
$\gamma([0,1]) \subset \tilde \psi(U\times V)$.
 Define $I$ to be the set of points $t\in [0,1]$ such that there exists a flat 
chart 
 $\psi:\mathcal U\times \mathcal V\rightarrow\mathcal G\times \mathcal M$ with 
 $\gamma([0,t])\subset \tilde \psi(U\times V)$.
 The set $I$ is open and $I\neq [0,1]$ by assumption.
 Let $s$ be the minimum of $[0,1]\setminus I$. 
 There is a flat 
 chart $\psi_1:\mathcal U_1\times\mathcal V_1\rightarrow \mathcal G\times 
\mathcal M$ with respect to 
 $(\mathcal D, U_1, V_1, \pi_G(\gamma(s)), \mathrm{id})$ with 
 $\gamma(s)\in \tilde \psi_1(U_1\times V_1)$. 
 By the choice of $s$ there is $t\in [0,s)$ with $\gamma([t,s])\subset \tilde 
\psi_1(U_1\times V_1)$ 
 such that there is a flat chart 
 $\psi_2:\mathcal U_2\times\mathcal V_2\rightarrow \mathcal G\times \mathcal M$ 
with respect to
 $(\mathcal D, U_2, V_2, \pi_G(\gamma(0)),\mathrm{id})$ with $\gamma([0,t])\in 
\tilde\psi_2(U_2\times V_2)$.
 Then $h=\pi_G(\gamma(t))\in U_1\cap U_2$ 
 and after possibly shrinking $V_2$ there exists a diffeomorphism $\rho$ such 
that 
 $\psi_1\circ(\mathrm{id}\times \rho)$ is a flat chart with respect to 
 $(\mathcal D, U_1, V_2, h,\rho')$ and $\psi_2$ with respect to $(\mathcal D, 
U_2, V_2, h, \rho')$ 
 for some $\rho':\mathcal V_2\rightarrow\mathcal M$.
 Therefore, $\psi_1\circ (\mathrm{id}\times \rho)$ and $\psi_2$ agree on their 
common domain of definition and they define 
 a flat chart $\psi:(\mathcal U_1\cup \mathcal U_2)\times V_2\rightarrow 
\mathcal G\times \mathcal M$
 with $\gamma([0,s])\subset \tilde \psi((U_1\cup U_2)\times V_2)$, contradicting 
the definition of $s$.
\end{proof}

The preceding proposition allows us to glue together flat charts in the case of 
a univalent infinitesimal action.
This also implies the next corollary.

\begin{cor}\label{cor: univalence and compact subsets}
 The infinitesimal action $\lambda$ is univalent if and only if 
 for any compact subset $\Sigma'$ of a leaf $\Sigma\subset G\times M$ there
 exists a flat chart $\psi:\mathcal U\times\mathcal V\rightarrow\mathcal 
G\times\mathcal M$
 with $\Sigma'\subset \tilde\psi(U\times V)$.
\end{cor}

In the following the structure of the distribution $\mathcal D$ associated to 
$\lambda$ is investigated further. 
Let $\Sigma\subset G\times M$ be a leaf, $(g,p)\in \Sigma$ and 
$\gamma:[0,1]\rightarrow \Sigma$ a closed path 
with $\gamma(0)=\gamma(1)=(g,p)$ and let $\gamma_G=\pi_G\circ \gamma$. 
We now want to associate a germ of a local diffeormorphism of $\mathcal M$ 
around $p$ measuring
the ``holonomy'' along the path~$\gamma$.

To do so, let $0=t_0<\ldots<t_k=1$ be a partition of $[0,1]$ such that there are 
flat charts
$\psi_i:\mathcal U_i\times \mathcal V\rightarrow \mathcal G\times \mathcal M$, 
$i=0,\ldots,k-1$,
with $\gamma([t_i,t_{i+1}])\subset \tilde \psi _i(U_i\times\{p\})$, such that 
$\psi_i$ and $\psi_{i+1}$ coincide on their common domain of definition
and such that $\psi_0$ is a flat chart with respect to $(\mathcal D, 
U_0,V,g,\mathrm{id})$.
We have $\tilde \psi _i(U_i\times \{p\})\subset \Sigma$ for any $i$.
By Lemma~\ref{lemma: almost flat chart} there is a diffeomorphism $\rho:\mathcal 
V\rightarrow \mathcal M$ onto its 
image such that $\psi_{k-1}$ is a flat chart with respect to $(\mathcal D, 
U_{k-1}, V, g=\gamma_G(1),\rho)$.
Since $(g,p)=\gamma(1)\in \tilde \psi _{k-1}(U_{k-1}\times \{p\})$, we have 
$(g,\tilde \rho (p))=\tilde \psi _{k-1} (g,p)=(g,p)$ and thus 
$\tilde \rho (p)=p$.

Define $\Phi(\gamma)$ to be the germ of the local diffeomorphism $\rho$ in $p$.
The local uniqueness of flat chart implies that $\Phi(\gamma)$ does not depend 
on the actual choice of the flat charts 
$\psi_i:\mathcal U_i\times \mathcal V_i\rightarrow \mathcal G\times \mathcal M$.
Let $\mathrm{Diff}_p(\mathcal M)$ denote the set of germs of local 
diffeomorphisms $\chi:\mathcal V_1\rightarrow \mathcal V_2$ in $p\in M$, where 
$\mathcal V_i=(V_i,\mathcal O_\mathcal M)$, $i=1,2$, are open subsupermanifolds 
of $\mathcal M$ with $p\in V_i$.
Then $\Phi(\gamma)$ is an element of $\mathrm{Diff}_p(\mathcal M)$ for each 
closed path $\gamma:[0,1]\rightarrow \Sigma$ with 
$\gamma(0)=\gamma(1)=(g,p)$.

In the case complex supermanifolds and holomorphic maps 
$\mathrm{Diff}_p(\mathcal M)$ should be replaced by $\mathrm{Hol}_p(\mathcal 
M)$, the 
set of germs of local biholomorphisms $\chi:\mathcal V_1\rightarrow\mathcal 
V_2$.

\begin{prop}
 The germ $\Phi(\gamma)$ only depends on the homotopy class $[\gamma]$ of the 
closed path $\gamma$ with $\gamma(0)=\gamma(1)=(g,p)$. 
 Therefore, the assignment 
 $[\gamma]\mapsto \Phi([\gamma])=\Phi(\gamma)$ defines a maps
 $$\Phi=\Phi_\Sigma=\Phi_{\Sigma,(g,p)}:\pi_1(\Sigma,(g,p))\rightarrow 
\mathrm{Diff}_p(\mathcal M).$$
\end{prop}

\begin{proof}
 Let $\gamma_s:[0,1]\rightarrow \Sigma$, $s\in [0,1]$, be a continuous family of 
closed paths with $\gamma_s(0)=\gamma_s(1)=(g,p)$ and 
 $\gamma_0=\gamma$.
 Let $s_0\in [0,1]$, $0=t_0<\ldots<t_k=1$ and $\psi_i:\mathcal U_i\times 
\mathcal V_i\rightarrow\mathcal G\times\mathcal M$, $i=0,\ldots,k-1$, flat 
charts
 with $\gamma_{s_0}([t_i,t_{i+1}])\subset \tilde \psi _i(U_i\times V_i)$ and 
such that $\psi_0$ is a flat chart with respect to 
 $(\mathcal D, U_0,V_0,g,\mathrm{id})$ and $\psi_i$ and $\psi_{i+1}$ coincide on 
their common domain of definition.
 Then $\psi_{k-1}$ is a flat chart with respect to $(\mathcal D, U_{k-1}, 
V_{k-1}, g, \rho)$ for some local diffeomorphism $\rho$ around $p$ and 
 $\Phi(\gamma_{s_0})$ is the germ of $\rho$ in $p$. 
 Since all intervals $[t_i,t_{i+1}]$ are compact there exists an open 
neighbourhood $J\subseteq [0,1]$ of $s_0$
 such that $\gamma_s([t_i,t_{i+1}])\subset \tilde \psi _i(U_i\times V_i)$ for 
$i=0,\ldots,k-1$ and all $s\in J$. 
 Therefore, $\Phi(\gamma_{s_0})=\Phi(\gamma_s)$ for all $s\in J$ and the set 
$\{s\in [0,1]|\,\Phi(\gamma_s)=\Phi(\gamma)\}$ is open and closed. Since 
$\gamma=\gamma_0$ we 
 get $\Phi(\gamma_s)=\Phi(\gamma)$ for all $s\in [0,1]$.
\end{proof}

\begin{rmk}
 The definition of $\Phi$ implies that it is a group homomorphism: If $\gamma$ 
is a constant path, then $\Phi([\gamma])$ is the germ of the identity 
 $\mathrm{id}:\mathcal M\rightarrow \mathcal M$ and we have
 $\Phi([\gamma_1]\cdot[\gamma_2])=\Phi([\gamma_1])\circ\Phi([\gamma_2])$ for the 
composition $\gamma_1\cdot\gamma_2$ of two closed paths 
 $\gamma_1$ and $\gamma_2$.
\end{rmk}

\begin{rmk}\label{rmk: holonomy-even part}
 By Lemma~\ref{lemma: even part of a flat chart} and Remark~\ref{rmk: equality 
of flat charts and even parts} the morphism $\Phi$ does not 
 depend on whether its construction is done with respect to the distribution 
$\mathcal D$ on $\mathcal G\times \mathcal M$
 associated to $\lambda$ or to the distribution $\mathcal D_0$ on 
$G\times\mathcal M$ associated to 
 the infinitesimal action $\lambda_0=\lambda|_{\g_0}$.
\end{rmk}

\begin{rmk}
 Due to the connectedness of the leaves $\Sigma$ the map $\Phi_{\Sigma,(g,p)}$ 
is trivial for some $(g,p)\in \Sigma$
 if and only if $\Phi_{\Sigma, (h,q)}$ is trivial for all $(h,p)\in \Sigma$.
 
 The triviality of the map $\Phi=\Phi_\Sigma=\Phi_{\Sigma,(g,p)}$ can be viewed 
as a sort of absence of holonomy for the leaf $\Sigma$.
\end{rmk}

\begin{ex}\label{ex: S^1-example}
 Let $\mathcal G=S^1$, with coordinate $\phi$, and $\mathcal M=\R^{0|2}$, with 
coordinates $\theta_1,\theta_2$.
 Let $X=\theta_1\frac{\partial}{\partial \theta_2}$ and consider the 
infinitesimal $S^1$-action $\lambda:\mathrm{Lie}(S^1)\cong 
\R\rightarrow\mathrm{Vec}(\mathcal M)$, $\lambda(t)=tX$.
 The unique leaf of the distribution $\mathcal D$ on $S^1\times\mathcal M$, 
spanned by $\frac{\partial}{\partial \phi}+X$,
 is $\Sigma=S^1\times \{0\}=S^1\times M$.
 Let $\phi_0\in S^1$ and $r:\Omega\rightarrow \R$ be a local inverse around 
$1\in \Omega\subset S^1$ 
 of $\R\rightarrow S^1$, $t\mapsto e^{it}$.
 Then 
 
$$\psi^*(\phi,\theta_1,\theta_2)=(\phi,\rho^*(\theta_1),\rho^*(\theta_2))+(0,0,
r(\phi {\phi_0}^{-1})\rho^*(\theta_1))$$
 defines the pullback of a flat chart $\psi$ with respect to $(\mathcal D, U, V, 
\phi_0,\rho)$ 
 for $V=\{0\}=M$, $U\subset S^1$ with $\phi_0\in U$ and $U{\phi_0}^{-1}\subseteq 
\Omega$ and a diffeormorphism $\rho:\R^{0|2}\rightarrow \R^{0|2}$.
 For arbitrary $\phi_0'\in U$ the map $\psi$ is also a flat chart with respect 
to 
 $(\mathcal D, U, V, \phi_0',\rho')$ for $\rho':\R^{0|2}\rightarrow \R^{0|2}$ 
with pullback
 
$$({\rho'})^*(\theta_1,\theta_2)=\rho^*(\theta_1,\theta_2)+(0,r(\phi_0'{\phi_0}^
{-1})\rho^*(\theta_1))
 =\rho^*(\theta_1,\theta_2)+\left(0,\bigg(\int_{\phi_0}^{\phi_0'} 1 
d\phi\bigg)\rho^*(\theta_1)\right),
 $$
 where the integral might be taken along any path in $\Omega$. In particular 
$(\rho')^*(\theta_1)=\rho^*(\theta_1)$.
 
 A calculation shows that the map $\Phi:\pi_1(\Sigma,(\phi_0,0))\rightarrow 
\mathrm{Diff}_0(\R^{0|2})=\mathrm{Diff}(\R^{0|2})$
 is given by
 $$\Phi([\gamma])^*(\theta_1,\theta_2)=\mathrm{id}^*(\theta_1,\theta_2)
 +\left(0,\bigg(\int_{\gamma}1 d\phi\bigg)\theta_1\right)
 =\left(\theta_1,\theta_2+\bigg(\int_\gamma 1 d\phi\bigg)\theta_1\right)$$
 for any $\gamma:[0,1]\rightarrow S^1\cong\Sigma$, $\gamma(0)=\gamma(1)=\phi_0$.
 Thus, identifying $\pi_1(\Sigma,(\phi_0,0))$ and $\Z$, we get
 $$\Phi:\Z\rightarrow 
\mathrm{Diff}(\R^{0|2}),\,\Phi(k)^*(\theta_1,\theta_2)=(\theta_1,\theta_2+2\pi 
k\theta_1).$$
\end{ex}

We now establish an equivalence between univalent infinitesimal actions and 
infinitesimal actions with univalent reduced infinitesimal action
and leaves without holonomy.

\begin{prop}\label{prop: univalence}
 The infinitesimal action $\lambda$ is univalent if and only if the reduced 
infinitesimal action 
 $\tilde \lambda :\g_0\rightarrow \mathrm{Vec}(M)$ is univalent and for all 
leaves $\Sigma$ the map
 $\Phi_\Sigma$ is trivial.
\end{prop}

\begin{proof}
 If the infinitesimal action $\lambda$ is univalent, then the reduced 
infinitesimal action $\tilde \lambda$ is univalent as a direct 
 consequence. Moreover, for any closed path $\gamma:[0,1]\rightarrow \Sigma$ 
there is a flat chart $\psi:\mathcal U\times \mathcal V\rightarrow \mathcal 
G\times \mathcal M$ with $\gamma([0,1])\subset \tilde\psi (U\times V)$ and thus 
$\Phi([\gamma])=\mathrm{id}$.
 
 Now, suppose that $\tilde \lambda$ is univalent and $\Phi_{\Sigma}$ is trivial 
for all leaves $\Sigma\subset G\times M$.
 We need to show that any two
 flat charts $\psi_i:\mathcal U_i\times \mathcal V_i\rightarrow\mathcal 
G\times\mathcal M$, $i=1,2$, 
 with respect to $(\mathcal D, U_i,V_i,g,\rho_i)$ with $\rho_1=\rho_2$ coincide 
on their 
 common domain of definition (cf. Proposition~\ref{prop: univalence and 
uniqueness of flat charts}).
 By replacing $\psi_i$ by $\psi_i\circ (\mathrm{id}\times {\rho_i}^{-1})$ it is 
enough to show $\psi_1=\psi_2$ in 
 the case of $\rho_1=\rho_2=\mathrm{id}$.
 
 Let $(h,p)\in (U_1\times V_1)\cap (U_1\times V_1)=(U_1\cap U_2)\times (V_1\cap 
V_2)$ be arbitrary. 
 We have $\tilde \psi_i(h,p)\in\tilde \psi_i(U_i\times\{p\})\subset 
\Sigma_{(g,p)}$ 
 and thus $\tilde \psi_1(h,p)=\tilde \psi_2(h,p)$ since 
$\pi_G|_{\Sigma_{(g,p)}}$ is injective.
 After shrinking $V_2$ there is a diffeomorphism $\rho:\mathcal V_2\rightarrow 
\mathcal M$ such that 
 $\psi_2\circ(\mathrm{id}\times \rho)$ is defined and coincides with $\psi_1$ 
near $(h,p)$.
 Note that $\tilde\rho(p)=p$ since $\tilde\psi_1(h,p)=\tilde\psi_2(h,p)$.
 The composition $\psi_2\circ(\mathrm{id}\times \rho)$ is a flat chart with 
respect to $(\mathcal D, U_2, V_2, g, \rho)$.
 Let $\alpha:[0,1]\rightarrow \Sigma_{(g,p)}$ be a closed path with 
$\alpha(0)=\alpha(1)=(g,p)$, 
 $\alpha(\frac 1 2)=\tilde \psi_1(h,p)=\tilde 
\psi_2(h,\tilde\rho(p))=\tilde\psi_2(h,p)$ and 
 $\alpha([0,\frac 1 2])\subset \tilde \psi_1(U_1\times \{p\})$
 and $\alpha([\frac 1 2,1])\subset \tilde \psi_2(U_2\times \{p\})=\tilde \psi_2( 
(\mathrm{id}\times\tilde\rho)(U_2\times\{p\}))$. 
 By the definition of $\Phi=\Phi_{\Sigma_{(g,p)}}$
 we have $\Phi([\alpha])=\rho$. The triviality of $\Phi$ then gives 
$\rho=\mathrm{id}$ so that
 $\psi_1$ and $\psi_2$ agree near $(h,p)$. 
\end{proof}

\begin{cor}\label{cor: univalence-even part}
 The infinitesimal action $\lambda$ is univalent if and only if its 
 restriction $\lambda_0=\lambda|_{\g_0}$ is univalent.
\end{cor}
\begin{proof}
 The statement follows from the preceding propostion and the observation 
formulated in Remark~\ref{rmk: holonomy-even part}. 
\end{proof}

\subsection{The action on $\mathcal G\times \mathcal M$ from the right and 
invariant functions}
In this section, the action $R:\mathcal G\times(\mathcal G\times \mathcal 
M)\rightarrow \mathcal G\times \mathcal M$
of the Lie supergroup~$\mathcal G$ on the product $\mathcal G\times \mathcal M$ 
from the right,
which is in the classical case given by $(g,(h,x))\mapsto (hg^{-1},x)$, is 
introduced.
Its behaviour with respect to the distribution $\mathcal D$ associated to 
$\lambda:\g\rightarrow \mathrm{Vec}(\mathcal M)$ and in particular $\mathcal 
D$-invariant functions
is studied.

\begin{defi}
 Let $\mathcal O_{\mathcal G\times \mathcal M}^\mathcal D$ be the sheaf of 
$\mathcal D$-invariant functions on $\mathcal G\times \mathcal M$,
 i.e. $$\mathcal O_{\mathcal G\times\mathcal M}^\mathcal D(\Omega)
 =\{f\in\mathcal O_{\mathcal G\times\mathcal M}(\Omega)|\,\mathcal D\cdot 
f=0\}$$ 
 for any open subset $\Omega\subset G\times M$.
\end{defi}

\begin{rmk}\label{rmk: invariant functions}
 For any $\mathcal D$-invariant function $f$ on $\mathcal G\times\mathcal M$, 
the underlying 
 function $\tilde f$ is constant along the leaves, and in the classical case 
every such function is $\mathcal D$-invariant.
 
 Moreover, if $\psi:\mathcal U\times \mathcal V\rightarrow \mathcal G\times 
\mathcal M$ is a flat chart
 and $f$ is $\mathcal D$-invariant, then $\psi^*(f)$ is $\mathcal D_\mathcal 
G$-invariant
 since $\psi_*(\mathcal D_\mathcal G)=\mathcal D$. The $\mathcal D_\mathcal 
G$-invariant functions on
 $\mathcal U\times\mathcal V\subseteq\mathcal G\times\mathcal M$ are of the form 
$f_\mathcal M=1\otimes f_\mathcal M$ for 
 $f_\mathcal M\in \mathcal O_\mathcal M(V)$ so that $\psi^*(f)=1\otimes 
f_\mathcal M$ for an appropiate $f_\mathcal M$.
\end{rmk}

We have the following identity principle for $\mathcal D$-invariant functions on 
$\mathcal G\times \mathcal M$:

\begin{lemma}\label{lemma: uniqueness of invariant functions}
 Let $W\subseteq G\times M$ be open, $\Sigma\subset G\times M$
 be a leaf with $\Sigma \subset W$ and let $f_1,f_2\in \mathcal O_{\mathcal 
G\times \mathcal M}^\mathcal D (W)$.
 If $f_1=f_2$ on an open neighbourhood of some $x\in \Sigma$, then 
 $f_1$ and $f_2$ coincide on an open neighbourhood of the leaf $\Sigma$.
\end{lemma}

\begin{proof}
 Define $\Sigma'\subseteq \Sigma$ to be the subset of points $y\in \Sigma$ such 
$f_1$ and $f_2$ are equal on some
 open neighbourhood of $y$. The set $\Sigma'$ is open and contains $x$ by 
assumption.

 For any flat chart $\psi:\mathcal U\times\mathcal V\rightarrow \mathcal 
G\times\mathcal M$ with $\tilde \psi(U\times V)\subset W$
 we have $\psi^*(f_i)=1\otimes f_{\mathcal M,i}$ for appropriate $f_{\mathcal 
M,i}$ by Remark~\ref{rmk: invariant functions}.
 Therefore, $f_1=f_2$ near $\tilde \psi(U\times\{q\})$ if and only if 
$f_{\mathcal M,1}=f_{\mathcal M,2}$ near $q\in M$.
 It follows that the set $\Sigma'$ is also closed and thus equal to $\Sigma$. 
\end{proof}

\begin{defi}
 Let $\mu:\mathcal G\times\mathcal G\rightarrow \mathcal G$ denote the 
multiplication of $\mathcal G$, $\iota:\mathcal G\rightarrow \mathcal G$ 
 the inversion and $\tau:\mathcal G\times\mathcal G\rightarrow \mathcal 
G\times\mathcal G$
 the morphism which interchanges the two components. 
 Then define the action of $\mathcal G$ on itself from the right as
 $r:\mathcal G\times\mathcal G\rightarrow \mathcal 
G,\,r=\mu\circ\tau\circ(\iota\times\mathrm{id}),$
 whose underlying action is given by $(g,h)\mapsto hg^{-1}$.
 Define now a $\mathcal G$-action on $\mathcal G\times\mathcal M$ by
 $$R:\mathcal G\times (\mathcal G\times\mathcal M)\rightarrow\mathcal G\times 
\mathcal M,\,R=r\times\mathrm{id}.$$
\end{defi}

\begin{lemma}
 For every right-invariant vector field $X$ on $\mathcal G$ we have 
 
$$(\mathrm{id}^*\otimes(X\otimes\mathrm{id}^*+\mathrm{id}
^*\otimes\lambda(X)))\circ R^*
 =R^*\circ(X\otimes\mathrm{id}^*+\mathrm{id}^*\otimes\lambda(X)).$$
\end{lemma}
 
\begin{proof}
 The right-invariance of $X$ is equivalent to $\mu^*\circ X=(X\otimes 
\mathrm{id}^*)\circ \mu^*$ and a short calculation yields 
 $(\mathrm{id}^*\otimes X)\circ r^*=r^*\circ X$, which directly implies the 
desired equality.
\end{proof}

\begin{cor}\label{cor: right-multiplication and invariant functions}
 We have $$R^*(\mathcal O_{\mathcal G\times\mathcal M}^\mathcal D)
 \subset \tilde R_*(\mathcal O_{\mathcal G\times\mathcal G\times \mathcal 
M}^{\id\otimes\mathcal D}),$$
 where $\id\otimes \mathcal D$ is the distribution on $\mathcal G\times(\mathcal 
G\times\mathcal M)$
 spanned by vector fields of the form $\mathrm{id}^*\otimes Y$ for vector fields 
$Y$ belonging to $\mathcal D$
 and $\mathcal O_{\mathcal G\times\mathcal G\times \mathcal 
M}^{\id\otimes\mathcal D}$ denotes the sheaf of 
 $\id\otimes\mathcal D$-invariant functions.
\end{cor}

\begin{proof}
Using the preceding lemma, we have
$$(\mathrm{id}^*\otimes(X\otimes\mathrm{id}^*+\mathrm{id}^*\otimes 
\lambda(X)))(R^*(f))
=R^*\circ(X\otimes\mathrm{id}^*+\mathrm{id}^*\otimes\lambda(X))(f)=R^*(0)=0$$
for any $\mathcal D$-invariant function $f$ on $\mathcal G\times \mathcal M$ and 
$X\in \g$.
Thus, $R^*(f)$ is $\id\otimes \mathcal D$-invariant.
\end{proof}

\begin{defi}
 For $g\in G$, let $r_g:\mathcal G\rightarrow \mathcal G$ denote the composition
 of the action~$r$ and the inclusion
 $\mathcal G\hookrightarrow\{g\}\times\mathcal G\subset\mathcal G\times\mathcal 
G$,
 and define 
 $$R_g:\mathcal G\times\mathcal M\rightarrow\mathcal G\times\mathcal 
M,\,R_g=(r_g\times\mathrm{id}).$$
Since $r$ is an action, $r_g$ and $R_g$ are diffeomorphisms, 
$(r_g)^{-1}=r_{g^{-1}}$ and $(R_g)^{-1}=R_{g^{-1}}$.
\end{defi}

\begin{lemma}
 Let $\iota_h:\mathcal M\hookrightarrow\{h\}\times\mathcal M
 \subset\mathcal G\times \mathcal M$ denote again the canonical inclusion for 
any $h\in G$.
 For each $g\in G $ the map $R_g:\mathcal G\times \mathcal M\rightarrow \mathcal 
G\times \mathcal M$ satisfies
\begin{enumerate}[(i)]
 \item $(R_g)_*(\mathcal D_\mathcal G)=\mathcal D_\mathcal G$, and
 \item $(R_g)_*(\mathcal D)=\mathcal D$.
\end{enumerate}
\end{lemma}
\begin{proof}
 Property $(i)$ can be directly obtained from the definition of $R_g$. 
Property~$(ii)$ follows from the fact
that $(r_g)_*(X)=X$ for every right-invariant vector field $X$, and therefore 
$(R_g)_*(X+\lambda(X))=(r_g)_*(X)+\lambda(X)=X+\lambda(X)$ so that 
$(R_g)_*(\mathcal D)=\mathcal D$.
\end{proof}

The composition of flat charts and maps of the form $R_g$ exhibits a special 
behaviour as specified in the following lemma.

\begin{lemma}\label{lemma: symmetry of flat charts}
 Let $g\in G$ and let $\psi:\mathcal U\times \mathcal V\rightarrow \mathcal 
G\times \mathcal M$ be a flat chart with respect to 
 $(\mathcal D, U, V, h, \rho)$.
Then the composition 
$$\psi'=R_{g^{-1}}\circ\psi\circ( R_g|_{Ug\times V}):\mathcal Ug\times \mathcal 
V\rightarrow \mathcal G\times \mathcal M$$ is a flat chart 
with respect to $(\mathcal D, Ug,V,hg,\rho)$ where $Ug=\{ug|\,u\in U\}$ and 
$\mathcal{U}g=(Ug,\mathcal O_\mathcal G|_{Ug})$.
\end{lemma}

\begin{proof}
 We have $ (\psi')_*(\mathcal D_\mathcal G)
=\mathcal D$ using $(R_g)_*(\mathcal D_\mathcal G)=\mathcal D_\mathcal G$ and 
$(R_g)_*(\mathcal D)=\mathcal D$.
Moreover, direct calculations show
$\pi_\mathcal G\circ \psi'=\pi_\mathcal G $
and 
$\psi'\circ \iota_{hg}=\iota_{hg}\circ \rho.$
\end{proof}

\begin{cor}
 The underlying classical Lie group $G$ acts on the space of leaves 
$\Sigma\subset G\times M$
 by $(g,\Sigma)\mapsto \tilde R_g(\Sigma)$.
 For $(h,p)\in G\times M$ we have $\tilde 
R_g(\Sigma_{(h,p)})=\Sigma_{(hg^{-1},p)}$.
\end{cor}

\begin{proof}
 Since $R_g$ preserves the distribution $\mathcal D$, 
 $\tilde R_g$ maps leaves diffeomorphically onto leaves.
 We have $\tilde R_g(\Sigma_{(h,p)})=\Sigma_{(hg^{-1},p)}$ because 
$(hg^{-1},p)=\tilde R_g(h,p)\in \tilde R_g(\Sigma_{(h,p)})$.
\end{proof}

\begin{rmk}
 If there exists an action $\varphi:\mathcal G\times\mathcal M\rightarrow 
\mathcal M$ with infinitesimal action~$\lambda$,
 then $\tilde R_g(\Sigma_{(e,p)})=\Sigma_{(g^{-1},p)}=\Sigma_{(e,\tilde\varphi 
(g,p))}$ for any $p\in M$:
 
 \noindent Let $\psi=(\mathrm{id}\times\varphi)\circ(\mathrm{diag}\times 
\mathrm{id}):\mathcal G\times\mathcal M\rightarrow\mathcal G\times\mathcal M$.
 The map $\psi$ is a flat chart (cf. Lemma \ref{lemma: local action and flat 
charts}) and 
 $\tilde\psi(G\times\{p\})=\Sigma_{(e,p)}$. 
 We have $(e,\tilde\varphi(g,p))=\tilde R_g(g,\tilde\varphi(g,p))=\tilde 
R_g(\tilde\psi(g,p))\in \tilde R_g(\tilde \psi(G\times\{p\}))$
 and thus $\tilde 
R_g(\Sigma_{(e,p)})=\Sigma_{(g^{-1},p)}=\Sigma_{(e,\tilde\varphi (g,p))}$.
\end{rmk}

\begin{prop}\label{prop: univalent action and special leaves}
 The infinitesimal action $\lambda:\g\rightarrow\mathrm{Vec}(\mathcal M)$ is 
univalent 
 if and only if every leaf of the form $\Sigma_{(e,p)}$ for $p\in M$ is 
univalent.
\end{prop}

\begin{proof}
 If $\lambda$ is univalent, then all leaves are univalent, in particular each 
leaf of the form $\Sigma_{(e,p)}$.
 Assume now that all leaves $\Sigma_{(e,p)}$ are univalent. Let $\Sigma$ be an 
arbitrary 
 leaf and let $(g,p)\in \Sigma$. We have $\Sigma_{(e,p)}=\tilde 
R_g(\Sigma_{(g,p)})=\tilde R_g(\Sigma)$.
 If $\Omega\subset \Sigma=\Sigma_{(g,p)}$ is a relatively compact subset, then 
$\tilde R_{g}(\Omega)\subset \Sigma_{(e,p)}$ 
 is relatively compact and the univalence of $\Sigma_{(e,p)}$ yields the 
existence of a flat chart 
 $\psi:\mathcal U\times\mathcal V\rightarrow\mathcal G\times\mathcal M$ with
 $\tilde R_{g}(\Omega)\subset\tilde \psi (U\times V)$.
 By Lemma \ref{lemma: symmetry of flat charts} the map 
 $$R_{g^{-1}}\circ\psi\circ R_{g}: \mathcal U g\times\mathcal 
V\rightarrow\mathcal G\times \mathcal M$$
 is a flat chart and 
 $\Omega=\tilde R_{g^{-1}}(\tilde R_{g}(\Omega))\subset \tilde R_{g^{-1}}(\tilde 
\psi (U\times V))
 =(\tilde R_{g^{-1}}\circ \tilde\psi\circ\tilde R_{g})(Ug\times V).$
\end{proof}

\subsection{Globalizations of infinitesimal actions on supermanifolds}
We now study conditions for the existence of globalizations. 
The main result is the following:

\begin{thm}\label{thm: condition for globalizability}
 Let $\lambda:\g\rightarrow\mathrm{Vec}(\mathcal M)$ be an infinitesimal action. 
Then the following statements are equivalent:
 \begin{enumerate}[(i)]
  \item The infinitesimal action $\lambda$ is globalizable.
  \item The restricted infinitesimal action $\lambda_0=\lambda|_{\g_0}$ is 
globalizable.
  \item The infinitesimal action $\lambda$ is univalent.
  \item The reduced infinitesimal action $\tilde \lambda$ is univalent, i.e. 
$\pi_G|_\Sigma$ is injective for an arbitrary leaf $\Sigma$,
  and all leaves $\Sigma\subset G\times M$ are ``holonomy free'',
  i.e. the morphism $\Phi_\Sigma$ is trivial.
  \item There exists a local action $\varphi:\mathcal W\rightarrow \mathcal M$ 
with induced infinitesimal action~$\lambda$ whose domain 
  of definition is maximally balanced.
 \end{enumerate}
\end{thm}

 The equivalence of $(iii)$ and $(iv)$ is the content of Proposition~\ref{prop: 
univalence} and once the equivalence 
 of $(i)$ and $(iii)$ is established the equivalence of $(i)$ and $(ii)$ is a 
consequence of Corollary~\ref{cor: univalence-even part}.
 The other equivalences shall be proven 
 in the following: The implication $(iii)\Rightarrow (i)$ is the content of 
Proposition~\ref{prop: M^* supermanifold},
 $(i)\Rightarrow (v)$ follows from Proposition~\ref{prop: globalizability 
implies maximally balanced domain of definition},
 and $(v)\Rightarrow (iii)$ is proven in Proposition~\ref{prop: maximally 
balanced domain of definition implies univalence}.

\begin{rmk}[see \cite{Palais}, Chapter III, Theorem IV and Theorem V]
 In the classical case, Palais shows that for a univalent infinitesimal action 
the space of leaves $\Sigma\subset G\times M$ 
 of the distribution $\mathcal D$,
 which is denoted by $M^*=(G\times M)/\!\!\sim$, carries in a natural way the 
structure of a possibly non-Hausdorff manifold, 
 the action of $G$ on $G\times M$ by $g\cdot (h,p)=(hg^{-1},p)$ induces an 
action on the quotient space $M^*$ and 
 $M\rightarrow M^*$, $p\mapsto \Sigma_{(e,p)}$ is an injective embedding.
\end{rmk}

A similar construction is also important in the case of supermanifolds.

\begin{defi}\label{defi: definition of leaf space}
 Let $\lambda$ be an infinitesimal action of $\mathcal G$ on $\mathcal M$ and 
$\mathcal D$ the associated distribution on 
 $\mathcal G\times\mathcal M$.
 Define $$M^*=(G\times M)/\!\!\sim$$ to be the space of leaves $\Sigma\subset 
G\times M$ and 
 denote by $\tilde \pi:G\times M\rightarrow M^*$, $\tilde 
\pi(g,p)=\Sigma_{(g,p)}$, the projection. 
 Now endow $M^*$ with the quotient topology and define the sheaf $\mathcal 
O_{\mathcal M ^*}$ of $\mathbb Z_2$-graded algebras on $M^*$
 by setting  
 $$\mathcal O_{\mathcal M^*}=\tilde \pi_*(\mathcal O_{\mathcal G\times\mathcal 
M}^\mathcal D),$$ i.e. 
 $\mathcal O_{\mathcal M^*}(\Omega)=\mathcal O_{\mathcal G\times \mathcal 
M}^\mathcal D(\tilde\pi^{-1}(\Omega))$
 for any open subset $\Omega \subset M^*$, where $\mathcal O_{\mathcal G\times 
\mathcal M}^\mathcal D$ denotes again the 
 sheaf of $\mathcal D$-invariant functions on $\mathcal G\times\mathcal M$.
 Define the ringed space $$\mathcal M^*=(M^*,\mathcal O_{\mathcal M^*})$$ and
 let $\pi=(\pi^*,\tilde \pi):\mathcal G\times \mathcal M\rightarrow \mathcal 
M^*$,
 where $\pi^*:\mathcal O_{\mathcal M^*}\rightarrow \tilde\pi_*(\mathcal 
O_{\mathcal G\times \mathcal M})$ is given by the canonical inclusion
 $\mathcal O_{\mathcal M^*}(\Omega)=\mathcal O_{\mathcal G\times \mathcal 
M}^\mathcal D(\tilde\pi^{-1}(\Omega))
 \hookrightarrow \mathcal O_{\mathcal G\times\mathcal 
M}(\tilde\pi^{-1}(\Omega))$ for any open subset $\Omega\subseteq M^*$.
\end{defi}

\begin{defi}
 For any open subset $V\subseteq M$, $\mathcal V=(V,\mathcal O_\mathcal M|_V)$, 
and $g\in G$ we
 define a morphism of ringed spaces
 $$\iota_{\mathcal V,g}:\mathcal V\rightarrow \mathcal M^*\text{ by } 
\iota_{\mathcal V,g}=\pi\circ \iota_g,$$
 where $\iota_g:\mathcal V\hookrightarrow \{g\}\times \mathcal V\subset \mathcal 
G\times \mathcal M$
 denotes again the inclusion. The reduced map of $\iota_{\mathcal V, g}$ is 
given by 
 $p\mapsto \Sigma_{(g,p)}$.
 For $g=e$ and $\mathcal V=\mathcal M$, let
 $$\iota_\mathcal M=\iota_{\mathcal M,e}:\mathcal M\rightarrow \mathcal M^*,\, 
\iota_\mathcal M=\pi\circ \iota_e.$$ 
\end{defi}

\begin{rmk}
 Let $\psi:\mathcal U\times \mathcal V\rightarrow \mathcal G\times \mathcal M$ 
be a flat chart with
 respect to $(\mathcal D, U, V, g,\mathrm{id})$ and let
 $\pi_\mathcal M=\pi_\mathcal M|_{U\times V}:\mathcal U\times\mathcal 
V\rightarrow \mathcal V$ denote the projection onto the second component.
 As $\iota_g:\mathcal V\rightarrow \mathcal U\times \mathcal V$ is a section of 
$\pi_\mathcal M$ and 
 $\psi\circ\iota_g=\iota_g$, the diagram 
 $$\xymatrix{\mathcal U\times \mathcal V\ar[r]^\psi\ar[d]_{\pi_\mathcal M} & 
\mathcal G\times\mathcal M\ar[d]^\pi\\
 \mathcal V\ar[r]_{\iota_{\mathcal V,g}} &\mathcal M^*
 }$$ 
 is commutative.
\end{rmk}

We will see later on that $\mathcal M^*$ carries the structure of a 
supermanifold and $\iota_\mathcal M:\mathcal M\rightarrow \mathcal M^*$ is 
an open embedding 
if and only if the infinitesimal action $\lambda$ is globalizable. 
In this case $\mathcal M^*$ is itself a globalization of $\lambda$ and the 
$\mathcal G$-action on 
$\mathcal M^*$ is induced by the $\mathcal G$-action $R$ on $\mathcal 
G\times\mathcal M$.

\begin{rmk}
 The topological space $M^*$ only depends on the underlying infinitesimal action 
$\tilde\lambda:\g_0\rightarrow\mathrm{Vec}(M)$
 since $\widetilde{\mathcal D} =\mathcal D_{\tilde\lambda}$ (cf. 
Lemma~\ref{lemma: underlying distribution}).
 Hence, the map $\tilde\pi$ is an open map as in the classical 
 case~(cf. \cite{Palais}, Chapter~I, Theorem~III, or \cite{HeinznerIannuzzi}, 
\S~2, Proposition~2).
 
 The topological space $M^*$ fulfills the second axiom of countability because 
 $\tilde \pi$ is an open quotient map and $G\times M$ a manifold.
\end{rmk}

\begin{lemma}
 The space $\mathcal M^*=(M^*, \mathcal O_{\mathcal M^*})$ is a locally ringed 
space.
\end{lemma}

\begin{proof}
 For $f\in \mathcal O_{\mathcal M^*}(\Omega)=\mathcal O_{\mathcal G\times 
\mathcal M}^\mathcal D(\tilde \pi^{-1}(\Omega))$ 
 with $\Sigma\in \Omega\subseteq M^*$ denote by $[f]_\Sigma$ the germ of $f$ in 
the stalk 
 $(\mathcal O_{\mathcal M^*})_\Sigma$.
 We can define the ideal 
 $$\mathfrak m_\Sigma=\{[f]_\Sigma|\,\tilde f(\Sigma)=0\}\lhd (\mathcal 
O_{\mathcal M^*})_\Sigma$$
 for any leaf $\Sigma \in M^*$.
 Assume that $\mathfrak m_\Sigma$ is not a maximal ideal in the stalk $(\mathcal 
O_{\mathcal M^*})_\Sigma$.
 Then there is a proper ideal $I\lhd(\mathcal O_{\mathcal M^*})_\Sigma$ which is 
not contained in~$\mathfrak m_\Sigma$.
 Let $[f]_\Sigma\in I\setminus \mathfrak m_\Sigma$. Then we have $\tilde 
f(\Sigma)\neq 0$ and the continuity 
 of $\tilde f$ gives the existence of an open neighbourhood $\Omega$ of 
$\Sigma\in M^*$ such that 
 $f$ is defined on $\tilde\pi^{-1}(\Omega)$, 
 i.e. $$f\in \mathcal O_{\mathcal M^*}(\Omega)=\mathcal O_{\mathcal 
G\times\mathcal M}^\mathcal D(\tilde\pi^{-1}(\Omega))
 \subset \mathcal O_{\mathcal G\times\mathcal M}(\tilde\pi^{-1}(\Omega)),$$
 and $\tilde f(x)\neq 0$ for all $x\in \tilde\pi^{-1}(\Omega)$.
 Consequently, there exists $g\in \mathcal O_{\mathcal G\times\mathcal 
M}(\tilde\pi^{-1}(\Omega))$
 with $gf=fg=1$.
 The function $g$ is also $\mathcal D$-invariant since 
 $$0=Y(1)=Y(gf)=Y(g)f+(-1)^{|g||Y|}gY(f)=Y(g)f+0=Y(g)f$$ for all vector fields 
$Y$ belonging to $\mathcal D$
 and thus $Y(g)=0$.
 Therefore, $g$ defines an element in $\mathcal O_{\mathcal M^*}(\Omega)$ and 
$[g]_\Sigma [f]_\Sigma=[gf]_\Sigma=1\in I$ 
 which contradicts the assumption $I\neq (\mathcal O_{\mathcal M^*})_\Sigma$.
\end{proof}

\begin{lemma}
 The action of $\mathcal G$ on $\mathcal G\times \mathcal M$ from the right 
 $R:\mathcal G\times (\mathcal G\times\mathcal M)\rightarrow \mathcal G\times 
\mathcal M$ 
 induces an action $\chi$ of $\mathcal G$ on the space $\mathcal M^*$, i.e. a 
morphism of ringed spaces satisfying the usual action properties,
 such that 
 $$\xymatrix{\mathcal G\times (\mathcal G\times \mathcal 
M)\ar[r]^/.4em/R\ar[d]_{\mathrm{id}\times\pi}&\mathcal G\times\mathcal 
M\ar[d]^\pi\\
 \mathcal G\times \mathcal M^*\ar[r]_\chi& \mathcal M^*
 }$$ commutes.
 In particular, if $\mathcal M^*$ is a supermanifold, then $\chi$ is an action 
of the Lie supergroup $\mathcal G$ on $\mathcal M^*$.
\end{lemma}

\begin{proof}
 The underlying action~$\tilde \chi$ of $G$ on $M^*$ is given by 
 $$\tilde\chi:G\times M^*\rightarrow M^*,\,(g,\Sigma_{(h,p)})\mapsto \tilde 
R_g(\Sigma_{(h,p)})=\Sigma_{(hg^{-1},p)}.$$
 and continuous since $\tilde\pi\circ\tilde 
R=\tilde\chi\circ(\mathrm{id}_G\times \tilde\pi)$.

 If $f$ is any $\mathcal D$-invariant function on $\mathcal G\times \mathcal M$, 
then $R^*(f)$ is 
 $(\id\otimes \mathcal D)$-invariant (see Corollary~\ref{cor: 
right-multiplication and invariant functions}).
 Therefore, 
 $$(\pi\circ R)^*(f)\in (\mathcal O_{\mathcal G\times \mathcal G\times \mathcal 
M}^{\id\otimes\mathcal D})
 ((\tilde \pi\circ\tilde R)^{-1}(\Omega))\cong (\mathcal O_\mathcal 
G\hat{\otimes} \mathcal O_{\mathcal G\times \mathcal M}^\mathcal D)
 ((\tilde \pi\circ\tilde R)^{-1}(\Omega))$$
 for any $f\in \mathcal O_{\mathcal M^*}(\Omega)=\mathcal O_{\mathcal 
G\times\mathcal M}^\mathcal D(\tilde \pi^{-1}(\Omega))$, $\Omega\subseteq M^*$,
 and $(\pi\circ R)^*$ induces a morphism 
 $\chi^*:\mathcal O_{\mathcal M^*}\rightarrow \tilde \chi_*(\mathcal O_{\mathcal 
G\times \mathcal M^*})$ 
 with $R^*\circ \pi^*=(\mathrm{id}\times \pi)^*\circ \chi^*$. The map~$\chi$ 
defines an action of $\mathcal G$ on $\mathcal M^*$
 since $R$ is an action and $\chi$ inherits the respective properties.
\end{proof}

The infinitesimal action $$\lambda_\chi:\g\rightarrow \mathrm{Vec}(\mathcal 
M^*),\, \lambda_\chi(X)=(X(e)\otimes {\mathrm{id}}^*)\circ\chi^*$$
on the ringed space $\mathcal M^*$ extends the infinitesimal action $\lambda$ on 
$\mathcal M$ in the following sense:

\begin{lemma}\label{lemma: extension of infinitesimal action}
 For any $X\in \g$ we have 
 $$\lambda(X)\circ{\iota_\mathcal M}^*={\iota_\mathcal M}^*\circ 
\lambda_\chi(X).$$
\end{lemma}

\begin{proof}
 Let $\Omega\subseteq M^*$ be open and $f\in \mathcal O_{\mathcal M^*}(\Omega)$.
 Then $\pi^*(f)$ is $\mathcal D$-invariant on $\tilde\pi^{-1}(\Omega)\subset 
G\times M$ and thus 
 $(\mathrm{id}^*\otimes \lambda(X))(\pi^*(f))=-(X\otimes 
\mathrm{id}^*)(\pi^*(f))$ for $X\in\g$.
 Consequently, we have
 $$  (\lambda(X)\circ {\iota_\mathcal M}^*)(f)
  =(\lambda(X)\circ {\iota_e}^*)(\pi^*(f))
  ={\iota_e}^*( (\mathrm{id}^*\otimes \lambda(X))(\pi^*(f)))
  =-((X(e)\otimes \mathrm{id}^*) (\pi^*(f)).$$
 A calculation using the identities $\pi\circ R=\chi\circ(\mathrm{id}\times 
\pi)$ and $R\circ(\mathrm{id}\times \iota_e)=(\iota\times\mathrm{id})$ 
 gives 
 ${\iota_\mathcal M}^*\circ\lambda_\chi(X)=((-X(e))\otimes \mathrm{id}^*)\circ 
\pi^*$
 so that $\lambda(X)\circ{\iota_\mathcal M}^*={\iota_\mathcal M}^*\circ 
\lambda_\chi(X)$.
\end{proof}

\begin{prop}\label{prop: M^* supermanifold}
 Let $\lambda$ be univalent. Then $\mathcal M^*$ is a supermanifold, with a 
possibly non-Hausdorff underlying manifold $M^*$.
 
 Moreover, the morphism $\iota_\mathcal M:\mathcal M\rightarrow \mathcal M^*$ is 
an open embedding, 
 the infinitesimal action $\lambda_\chi$ induced by $\chi$
 extends the infinitesimal action of $\lambda$ on 
 $\mathcal M\cong (\tilde \iota_\mathcal M (M), \mathcal O_{\mathcal 
M^*}|_{\tilde \iota_\mathcal M (M)})$, and
 we have $G\cdot \tilde\iota_\mathcal M(M)=\tilde\chi(G\times 
\tilde\iota_\mathcal M(M))=M^*$.
 
 Thus $\mathcal M^*$ is a globalization of $\lambda$ and the infinitesimal 
action $\lambda$ is globalizable.
\end{prop}

The proof of the proposition makes use of the next lemma.

\begin{lemma}\label{lemma: extensions of invariant functions}
 Let $\lambda$ be univalent and $W\subseteq G\times M$ an open connected subset.
 For any $f\in \mathcal O_{\mathcal G\times\mathcal M}^\mathcal D(W)$
 there exists a unique extension
 $$\hat{f}\in \mathcal O_{\mathcal G\times\mathcal M}^\mathcal 
D(\tilde\pi^{-1}(\tilde\pi(W)))
 \text{ with }\hat f|_{W}=f.$$
\end{lemma}

\begin{proof}
Note that the open set $$\tilde \pi^{-1}(\tilde\pi(W))=\bigcup_{\substack{\Sigma 
\text{ leaf}\\ \Sigma\cap W\neq \emptyset}} \Sigma$$ is 
 again connected
 and by Lemma~\ref{lemma: uniqueness of invariant functions} a $\mathcal 
D$-invariant extension $\hat f$ of $f$ is unique.
 
 Let $\Sigma$ be a leaf with $\Sigma\cap W\neq \emptyset$, i.e. $\Sigma\in 
\tilde \pi(W)$, and let
 $\Sigma'\subset \Sigma$ be relatively compact. Since $\lambda$ is univalent, 
%(cf. Corollary~\ref{cor: univalence and compact subsets}),
 there exists a flat chart $\psi:\mathcal U\times\mathcal V\rightarrow\mathcal 
G\times \mathcal M$
 with $\Sigma'\subset \tilde\psi(U\times V)$ and $\tilde \psi(U'\times V)\subset 
W$ for some open subset $U'\subseteq U$.
 The function $(\psi|_{\tilde\psi^{-1}(W)})^*(f)$ is $\mathcal D_\mathcal 
G$-invariant and thus  of the form $(\psi|_{\tilde\psi^{-1}(W)})^*(f)=1\otimes 
f_\mathcal M$ on $U'\times V\subseteq \tilde\psi^{-1}(W)$
 for some $f_\mathcal M\in \mathcal O_\mathcal M(V)$.
 Now, $1\otimes f_\mathcal M$ is already defined on $U\times V$ and we define
 $\hat f$ on $\tilde \psi(U\times V)$ by 
 $$\hat f|_{\tilde \psi(U\times V)}=(\psi^{-1})^*(1\otimes f_\mathcal M).$$
 This yields a well-defined function $\hat f$ on $\tilde \pi^{-1}(\tilde\pi(W))$ 
due to the 
 uniqueness of flat charts following from the univalence of $\lambda$ (see 
Proposition~\ref{prop: univalence and uniqueness of flat charts}).
 Moreover, $\hat f$ is $\mathcal D$-invariant by construction and $\hat f|_W=f$.
\end{proof}

\begin{proof}[Proof of Proposition~\ref{prop: M^* supermanifold}]
 We prove that in the case of a univalent infinitesimal action the morphism 
$\iota_{\mathcal V,g}:\mathcal V\rightarrow\mathcal M^*$ 
 defines a chart for $\mathcal M^*$ if there is a flat chart $\psi:\mathcal 
U\times\mathcal V\rightarrow\mathcal G\times \mathcal M$
 with respect to $(\mathcal D, U, V, g, \mathrm{id})$. Due to the local 
existence of flat charts this implies that $\mathcal M^*$ is a supermanifold.
 
 By Proposition~\ref{prop: univalence}, the restriction $\tilde\pi_\mathcal 
G|_\Sigma:\Sigma\rightarrow G$ of the projection
 $\tilde \pi_\mathcal G:G\times M\rightarrow G$ is injective for all leaves 
$\Sigma\subset G\times M$. 
 Therefore, we have $\tilde \iota_{\mathcal 
V,g}(p)=\Sigma_{(g,p)}=\Sigma_{(g,q)}=\tilde \iota_{\mathcal V,g}(q) $ if and 
only if $p=q$.
 The map $\tilde\iota_{\mathcal V,g}$ is open because for any open subset 
$V'\subseteq V$
 the set $\tilde\iota_{\mathcal V,g}(V')=\tilde \iota_{\mathcal V,g}(\tilde 
\pi_\mathcal M(U\times V'))
 =\tilde \pi(\tilde \psi(U\times V))$ is open in $M^*$ using that $\psi$ is a 
local diffeomorphism and $\tilde\pi$ an open map.
 Consequently, $\tilde\iota_{\mathcal V,g}$ is a homeomorphism onto its image.
 
 To show that ${\iota_{\mathcal V, g}}^*$ is injective, let $\Omega\subseteq 
M^*$ be open and 
 $f_1,f_2\in \mathcal O_{\mathcal M^*}(\Omega)$.
 If ${\iota_{\mathcal V, g}}^*(f_1)={\iota_{\mathcal V,g}}^*(f_2)$, then also 
 $$\psi^*(\pi^*(f_1))={\pi_\mathcal M}^*({\iota_{\mathcal V, 
g}}^*(f_1))={\pi_\mathcal M}^*({\iota_{\mathcal V, 
g}}^*(f_2))=\psi^*(\pi^*(f_2)).$$
 Since $\pi^*:\mathcal O_{\mathcal M^*}=\tilde\pi_*(\mathcal O_{\mathcal G\times 
\mathcal M}^\mathcal D)\
 \hookrightarrow \tilde \pi_*(\mathcal O_{\mathcal G\times \mathcal M})$ is the 
canoncial inclusion,
 $\pi^*(f_i)$ and $f_i$ can be identified.
 If $\psi^*(f_1)=\psi^*(f_2)$, then $f_1=f_2$ on $\tilde \psi(\tilde 
\psi^{-1}(\tilde\pi^{-1}(\Omega)))$ and thus
 $f_1=f_2$ by Lemma~\ref{lemma: uniqueness of invariant functions}.
 
 For any $f_\mathcal V\in \mathcal O_{\mathcal M}|_V({\tilde\iota_{\mathcal 
V,g}}^{-1}(\Omega))$
 we have $1\otimes f_\mathcal V={\pi_\mathcal M}^*(f_\mathcal V)\in 
 \mathcal O_{\mathcal G\times \mathcal M}^{\mathcal D_\mathcal G}(U\times 
{\tilde\iota_{\mathcal V,g}}^{-1}(\Omega))
 =\mathcal O_{\mathcal G\times\mathcal M}^{\mathcal D_\mathcal G}(\tilde 
\psi^{-1}(\tilde\pi^{-1}(\Omega)))$.
 Thus $(\psi^{-1})^*(1\otimes f_\mathcal V)$ is a $\mathcal D$-invariant 
function on 
 $\tilde \psi(U\times V)\cap \tilde \pi^{-1}(\Omega)=\tilde \psi(\tilde 
\psi^{-1}(\tilde\pi^{-1}(\Omega)))$.
 By Lemma~\ref{lemma: extensions of invariant functions} there is a $\mathcal 
D$-invariant 
 extension $\hat f\in \mathcal O_{\mathcal G\times\mathcal M}^\mathcal 
D(\tilde\pi^{-1}(\Omega))=\mathcal O_{\mathcal M^*}(\Omega)$
 of $(\psi^{-1})^*(1\otimes f_\mathcal V)$ and we have ${\iota_{\mathcal V, 
g}}^*(\hat f)=f_\mathcal V$ 
 since $\psi^*(\pi^*(\hat f))=\psi^*((\psi^{-1})^*(1\otimes f_\mathcal 
V))=1\otimes f_\mathcal V$.
 Consequently, ${\iota_{\mathcal V, g}}^*$ is also surjective and thus 
$\iota_{\mathcal V, g}$ is a chart for $\mathcal M^*$.
 
 The morphism $\iota_\mathcal M=\iota_{\mathcal M, e}$ is an open embedding 
since 
 the univalence of $\lambda$ implies that $\tilde \iota_\mathcal M$ is injective 
with the same argument as for 
 $\tilde\iota_{\mathcal V,g}$ and locally $\iota_\mathcal M$ is of the form 
$\iota_{\mathcal V, e}$ such 
 that there exists a flat chart $\psi$ with respect to $(\mathcal D, U, V, e, 
\mathrm{id})$.
 
 By Lemma~\ref{lemma: extension of infinitesimal action}, the infinitesimal 
action $\lambda_\chi$ induced by $\chi$
 extends the infinitesimal action of $\lambda$ on 
 $\mathcal M\cong (\tilde \iota_\mathcal M (M), \mathcal O_{\mathcal 
M^*}|_{\tilde \iota_\mathcal M (M)})$.
 Since $$g\cdot \Sigma_{(e,p)}=\tilde \chi(g,\Sigma_{(e,p)})=\tilde 
R_g(\Sigma_{(e,p)})=\Sigma_{(g^{-1},p)}$$ 
 for any $p\in M$ and $\tilde \iota_\mathcal M(M)$ consists exactly of those 
leaves which intersect $\{e\}\times M$,
 we have $G\cdot\tilde \iota_\mathcal M(M) =\tilde \chi(G\times 
\tilde\iota_\mathcal M(M))=M^*$.
\end{proof}

\begin{lemma}\label{lemma: whole leaf}
 Let $\varphi:\mathcal W\rightarrow \mathcal M$ a local action with induced 
infinitesimal action 
 $\lambda$ and maximally balanced domain of definition.
 Then we have $\tilde \psi(W_p\times \{p\})=\Sigma_{(e,p)}$ for any $p\in M$ and
 $\psi=(\mathrm{id}\times \varphi)\circ (\mathrm{diag}\times 
\mathrm{id}):\mathcal W\rightarrow \mathcal G\times \mathcal M$, which is 
locally 
 a flat chart.
\end{lemma}

\begin{proof}
 The lemma follows from the analogous classical result (see \cite{Palais}, 
Chapter II, Theorem VI) since
 $\varphi:\mathcal W\rightarrow \mathcal M$ has a maximally balanced domain of 
definition if and only if the domain of definition $W$
 of the reduced local action is maximally balanced.
\end{proof}

\begin{prop}\label{prop: globalizability implies maximally balanced domain of 
definition}
  Let $\varphi':\mathcal G\times\mathcal M'\rightarrow \mathcal M'$ be a 
globalization of $\lambda$. 
  Then there is a local action $\varphi:\mathcal W\rightarrow \mathcal M$ with 
maximally balanced domain of definition
  and infinitesimal action $\lambda$.
  
  Moreover, $\varphi:\mathcal W\rightarrow \mathcal M$ is the unique maximal 
local action with infinitesimal action $\lambda$.
  Any two local actions $\varphi_i:\mathcal W_i\rightarrow \mathcal M$, $i=1,2$, 
with infinitesimal action $\lambda$ coincide
  on their common domain of definition and define a local action $\chi:\mathcal 
W_1\cup\mathcal W_2\rightarrow \mathcal M$ with 
  $\chi|_{W_1}=\varphi_1$ and $\chi|_{W_2}=\varphi_2$.
\end{prop}

\begin{proof}
 The set $(\tilde \varphi ')^{-1}(M)\cap (G\times M)$ is open in $G\times M$ and 
contains $\{e\}\times M$. 
 Let $W_p$ be the connected component of $e$ in $\{g\in G|\,\tilde \varphi 
'(g,p)\in M\}$ for $p\in M$ and 
 define $$W=\bigcup_{p\in M} W_p\times \{p\}.$$
 The set $W\subseteq G\times M$ is open and the largest domain of definition of 
a local action included in 
 $(\tilde \varphi ')^{-1}(M)\cap (G\times M)$ (cf. \cite{Palais}, Chapter II, 
Theorem I). 
 Let $\mathcal W=(W,\mathcal O_\mathcal W)$ and define $\varphi=\varphi 
'|_W:\mathcal W\rightarrow \mathcal M$. 
 The map $\varphi$ is a local action of $\mathcal G$ on $\mathcal M$.
 A direct calculation shows that the domain of definition $\mathcal W$ of 
$\varphi$ is maximally balanced. By the preceding lemma
 we have $\tilde \psi(W_p\times \{p\})=\Sigma_{(e,p)}$ for any $p\in M$, where 
 $\psi=(\mathrm{id}\times \varphi)\circ (\mathrm{diag}\times 
\mathrm{id}):\mathcal W\rightarrow \mathcal G\times \mathcal M$.
 
 Let $\chi:\mathcal W_\chi\rightarrow \mathcal M$ be any local action with 
induced infinitesimal action $\lambda$ and set
 $\psi_\chi=(\mathrm{id}\times \chi)\circ (\mathrm{diag}\times 
\mathrm{id}):\mathcal W_\chi \rightarrow \mathcal G\times \mathcal M$.
 For any $p\in M$, $\tilde \psi _\chi(W_{\chi,p}\times \{p\})$ is contained in 
the leaf $\Sigma_{(e,p)}$.
 Since $\Sigma_{(e,p)}=\tilde \psi(W_p\times\{p\})$, we get 
 $$W_{\chi,p}=\pi_G(\tilde\psi_\chi(W_{\chi,p}\times\{p\}))\subseteq 
\pi_G(\Sigma_{(e,p)})=\pi_G(\tilde\psi(W_p\times\{p\}))=W_p,$$
 which implies $W_\chi\subseteq W$. The uniqueness of local actions with a given 
domain of definition (see Corollary \ref{cor: uniqueness of local action}) 
yields $\varphi=\chi$ on $W_\chi$.
\end{proof}

\begin{prop}\label{prop: maximally balanced domain of definition implies 
univalence}
 Let $\varphi:\mathcal W\rightarrow \mathcal M$ be a local action with maximally 
balanced domain of definition. 
 Then its induced infinitesimal action is univalent.
\end{prop}

\begin{proof}
 Let $\psi=(\mathrm{id}\times \varphi)\circ (\mathrm{diag}\times 
\mathrm{id}):\mathcal W\rightarrow \mathcal G\times \mathcal M$ be the locally 
 flat chart associated to the local action $\varphi$. By Lemma \ref{lemma: whole 
leaf} we have $\tilde \psi(W_p\times\{p\})=\Sigma_{(e,p)}$ for any 
 $p\in M$.
 
 Let $\Omega\subset \Sigma_{(e,p)}$ be a relatively compact connected subset. By 
Lemma \ref{lemma: local action and flat charts} 
 there are subsets $U\subset G$ and $V\subset M$, $p\in V$, such that 
 $\psi|_{U\times V}$ is a flat chart and 
 $\Omega\subset \tilde\psi(U\times V)$. 
 Consequently, $\Sigma_{(e,p)}$ is univalent for any $p\in M$ and hence 
$\lambda$
 is univalent by Proposition~\ref{prop: univalent action and special leaves}
\end{proof}

\begin{ex}\label{ex: C-example}
 Let $\mathcal M=(\C\setminus \{0\})\times \C^{0|2}$, with coordinates 
$z,\theta_1,\theta_2$, 
 and let $\alpha:\C\setminus \{0\}\rightarrow \C$ be a holomorphic function.
 Consider the even holomorphic vector field $$X_\alpha=(1+\alpha(z) 
\theta_1\theta_2)\frac{\partial}{\partial z}$$
 on $\mathcal M$.
 We now examine for which $\alpha$ the infinitesimal $\C$-action 
$\lambda_\alpha$, $\lambda_\alpha(t)=tX_\alpha$, generated by $X_\alpha$ is 
globalizable.
 
 Let $\mathcal D_\alpha$ be the distribution spanned by 
$\frac{\partial}{\partial t}+X_\alpha$ on $\C\times\mathcal M$.
 The leaves $\Sigma\subset \C\times M$ of $\mathcal D_\alpha$ are of the 
 form $\Sigma=\Sigma_{(t,z)}=\{(t+s,z+s)|\,s\in \C\setminus\{-z\}\}$ for 
$(t,z)\in \C\times (\C\setminus\{0\})$.
 Each leaf $\Sigma$ is therefore biholomorphic to $\C\setminus\{0\}$.
 
 The reduced vector field $\tilde X_\alpha=\frac{\partial}{\partial z}$ always 
generates a globalizable infinitesimal action 
 and a globalization of $M=\C\setminus\{0\}$ is $M^*=\C$ with the usual addition 
as $\C$-action.
 If $\lambda_\alpha$ is globalizable, then the globalization 
 $\mathcal M^*_\alpha=(M^*,\tilde \pi_*\mathcal O_{\C\times \mathcal 
M}^{\mathcal D_\alpha})$ is a complex supermanifold of dimension $(1|2)$. 
 Every complex supermanifold $\mathcal N$ with underlying manifold $\C$ is split 
since $\C$ is Stein (see \cite{Onishchik}, Theorem 3.4). 
 Moreover, $\mathcal N$ is isomorphic to $\C^{1|n}$ for some $n\in \N$ since all 
holomorphic vector bundles on $\C$ are trivial.
 Therefore, $\mathcal M^*_\alpha$ is isomorphic to $\C^{1|2}$ if it is a 
supermanifold.

 An element $f=f_0+ f_1\theta_1+ f_2\theta_2+f_{12}\theta_1\theta_2\in \mathcal 
O_{\C\times \mathcal M}(\C\times M)$
 is $\mathcal D_\alpha$-invariant if and only if there
 exist holomorphic functions $g_i:\C\rightarrow \C$, $i=1,2,3$,
 with $$f_i(t,z)=g_i(z-t)$$
 and $f_{12}$ is locally of the form 
 $$f_{12}(t,z)=-A(z)\left(\frac{\partial}{\partial 
z}f_0\right)(t,z)+g_{12}(z-t),$$
 where $A$ is a primitive of $\alpha$ and $g_{12}$ a holomorphic function.
 There are two different cases:
 \begin{enumerate}[(i)]
  \item 
 If $\alpha$ has a global primitive $A:\C\setminus\{0\}\rightarrow \C$ 
 we have 
 \begin{align*} \mathcal O_{\mathcal 
M_\alpha^*}(M^*)&=\left\{g_0+g_1\theta_1+g_2\theta_2
 +\Big(g_{12}-A\frac{\partial}{\partial z}g_0\Big)\theta_1\theta_2 \Big|
 \, g_0,g_1,g_2,g_{12} \text{ holomorphic}\right\}\\
 &\cong \mathcal O_{\C^{1|2}}(\C).\end{align*}
 and it turns out that $\lambda_\alpha$ is globalizable
 with globalization $\mathcal M_\alpha^*\cong\C^{1|2}$, where $\C$ acts on 
$\C^{1|2}$ by the extension of the usual addition on $\C$, i.e.
 $(z,\theta_1,\theta_2)\mapsto(z+t,\theta_1,\theta_2).$
 Moreover, the open embedding $\iota_\mathcal M:\mathcal M\rightarrow\mathcal 
M_\alpha^*$
 can be realised as 
 $$\iota_\mathcal M:\C\setminus\{0\}\times \C^{0|2}\hookrightarrow \C^{1|2},\,
 {\iota_\mathcal 
M}^*(z,\theta_1,\theta_2)=(z-A(z)\theta_1\theta_2,\theta_1,\theta_2).$$
 
 \item 
 In the case where $\alpha$ does not have a global primitive $A$ on 
$\C\setminus\{0\}$, 
 we get 
 $$
 \mathcal O_{\mathcal M_\alpha^*}(M^*)
 =\left\{\lambda+g_1\theta_1+g_2\theta_2
 +g_{12}\theta_1\theta_2 \Big|
 \,\lambda \in \C,\,g_1,g_2,g_{12} \text{ holomorphic}\right\},$$
 which is not isomorphic to $\mathcal O_{\C^{1|2}}(\C)$. 
 Therefore, the ringed space $\mathcal M_\alpha^*$ is not a supermanifold and 
 $\lambda_\alpha$ is globalizable.
 \end{enumerate}
 For any $\alpha$ a flat chart $\psi$ with respect to 
 $(\mathcal D_\alpha, U, V,t_0,\rho)$ is given by the pullback 
 $$ \psi^*(t,z,\theta_1,\theta_2)\\
 =\Big(t,\rho^*(z,\theta_1,\theta_2)\Big)
 +\Big(0,(t-t_0)+\big(A(\tilde\rho(z)+(t-t_0))
 -A(\tilde\rho(z))\big)\rho^*(\theta_1\theta_2),0,0\Big),$$
 where $U\subseteq\C$ and $V\subseteq \C\setminus\{0\}$ need to be open subsets 
such that there
 exists a primitive $A$ of $\alpha$ on 
 $U+\tilde\rho(V)
 =\{t+z|\,t\in U,z\in \tilde \rho (V)\}\subset \C\setminus \{0\}$.
 
 We shall now explicitly describe the morphism
 $\Phi=\Phi_{\Sigma,{(s,z)}}:\pi_1(\Sigma,(s,z))\rightarrow 
 \mathrm{Hol}_z(\C\setminus\{0\}\times \C^{0|2})$
 for a leaf $\Sigma\subset \C\times M$ with $(s,z)\in \Sigma$.
 First remark that if $\psi$ is a flat chart with respect to 
 $(\mathcal D_\alpha,U, V, t_0,\rho)$ and 
 $t_0'\in U$, then $\psi$ is also a flat chart with respect to 
 $(\mathcal D_\alpha, U, V, t_0',\rho')$
 for 
 \begin{align*}
 &(\rho')^*(z,\theta_1,\theta_2)
 =\rho^*(z,\theta_1,\theta_2)
 +\Big((t_0'-t_0)+\big(A(\tilde\rho(z)+(t_0'-t_0))
 -A(\tilde\rho(z))\big)\rho^*(\theta_1\theta_2),0,0\Big)\\
 &=\rho^*(z,\theta_1,\theta_2)+\left(\int_{t_0}^{t_0'}1 dt
 +\bigg(\int_{t_0}^{t_0'}\alpha(\tilde\rho(z)
 +(t-t_0))dt\bigg)\rho^*(\theta_1\theta_2),0,0\right),
 \end{align*}
 where the integrals do not depend on the path if contained in U since $\alpha$ 
has a primitive on $U+\tilde\rho(V)$.
 If $\psi'$ is another flat chart with respect to $(\mathcal D, 
U',V',t_0',\rho')$ with $V\cap V'\neq \emptyset$
 and $t_0''\in U'$, then $\psi'$ is also a flat chart with respect to 
 $(\mathcal D, U',V', t_0'',\rho'')$ for 
 \begin{align*}
  (\rho'')^*(z,\theta_1,\theta_2)
  =\rho^*(z,\theta_1,\theta_2)
  +\left(\int_{t_0}^{t_0''}1 dt+\bigg(\int_{t_0}^{t_0''}\alpha(\tilde\rho(z)+
  (t-t_0))dt\bigg)\rho^*(\theta_1\theta_2),0,0\right),
 \end{align*}
 where the integrals need to be taken along appropiate paths.
 For any closed path $\gamma:[0,1]\rightarrow \Sigma$, 
$\gamma(r)=(\gamma_1(r),\gamma_2(r))$,
 $\gamma(0)=\gamma(1)=(s,z)$,
 we consequently get 
 $$\Phi([\gamma])^*(z,\theta_1,\theta_2) =\left(z+\bigg(\int_{\gamma_1}
 \alpha_{s,z} dt \bigg)\theta_1\theta_2,\theta_1,\theta_2\right)$$
 for $\alpha_{s,z}(t)=\alpha((z-s)+t)$.
 Thus, $\Phi$ is trivial if and only if $\alpha$ has a global primitive on 
$\C\setminus\{0\}$.
 Using Theorem~\ref{thm: condition for globalizability} this shows again that 
$\lambda_\alpha$ 
 is globalizable precisely if $\alpha$ has a global primitive.
 
 In the special case of $\alpha(z)=z^{-1}$, we have for example 
 $$\Phi:\Z\rightarrow \mathrm{Hol}_p(\C\setminus\{0\}\times \C^{0|2}),
 \,\Phi(k)^*(z,\theta_1,\theta_2)
 =\left(z+2\pi i k \theta_1\theta_2,\theta_1,\theta_2\right)$$
 for any leaf $\Sigma$, identifying the fundamental group of $\Sigma\cong 
\C\setminus \{0\}$ with $\Z$.
\end{ex}

\section{Actions of simply-connected Lie supergroups}
In this section a few consequences of the characterization of globalizable 
infinitesimal actions of
simply-connected Lie supergroups, i.e. Lie supergroups $\mathcal G$ whose 
underlying Lie group~$G$ is simply-connected, are given.

If $G$ is simply-connected and acts on the classical manifold $M$, the leaves 
$\Sigma\subset G\times M$ of the distribution associated to the infinitesimal 
action are all isomorphic to $G$ and thus simply-connected.
This yields consequences for the existence of globalizations of infinitesimal 
actions of 
simply-connected Lie supergroups
since there are no holonomy phenomena, i.e. the morphisms 
$\Phi_{\Sigma}:\pi_1(\Sigma)\rightarrow \mathrm{Diff}_p(\mathcal M)$ are all 
trivial, 
if the reduced infinitesimal action is global.

\begin{defi} 
 An infinitesimal action $\lambda:\g\rightarrow\mathrm{Vec}(\mathcal M)$ is 
called global if there is a $\mathcal G$-action on $\mathcal M$ which induces 
$\lambda$.
\end{defi}

\begin{rmk}[cf. \cite{Palais}, Chapter II, Section 4]
 Let $G$ be a Lie group acting on a manifold $M$. Then the leaves 
 $\Sigma\subset G\times M$ of the distribution 
 $\mathcal D_{\tilde\lambda}$ associated to the infinitesimal action 
 $\tilde \lambda$ induced by the $G$-action are
 all isomorphic to $G$.
\end{rmk}

As a consequence we obtain the following theorem.

\begin{thm}\label{thm: actions of simply-connected Lie supergroups}
 Let $\mathcal G$ be a simply-connected Lie supergroup and 
 $\lambda:\g\rightarrow \mathrm{Vec}(\mathcal M)$
 an infinitesimal action of $\mathcal G$ such that its reduced action 
 $\tilde\lambda:\g_0\rightarrow\mathrm{Vec}(M)$ 
 is global.
 Then the infinitesimal action $\lambda$ is globalizable, and $\mathcal M$ is 
 the unique globalization.
\end{thm}

\begin{proof}
 Let $\mathcal D$ denote again the distribution on $\mathcal G\times\mathcal M$ 
 associated to $\lambda$. 
 By Lemma~\ref{lemma: underlying distribution} we have 
 $\widetilde{\mathcal D}=\mathcal D_{\tilde\lambda}$
 and by definition the leaves of $\mathcal D$ are the leaves of 
 $\mathcal D_{\tilde\lambda}$.
 By the preceding remark all leaves $\Sigma\subset G\times M$ are isomorphic 
 to $G$ and consequently simply-connected. 
 Therefore, the morphisms 
 $\Phi:\pi_1(\Sigma,(g,p))\cong\{1\}\rightarrow \mathrm{Diff}_p(\mathcal M)$
 are all trivial. Since $\tilde\lambda$ is global and thus in particular 
globalizable, $\tilde\lambda$ is univalent.
 This implies that $\lambda$ is globalizable using the equivalent 
characterizations of globalizability formulated in 
 Theorem~\ref{thm: condition for globalizability}. 
 Let $\mathcal M'$ be a globalization of $\lambda$ and 
 $\iota_\mathcal M:\mathcal M\rightarrow \mathcal M'$ the open embedding.
 Then $M'$ is a globalization of $\tilde\lambda$ and because $M'$ is the unique 
globalization of
 $\tilde\lambda$ 
 (cf. \cite{Palais}, Chapter III, Theorem XII,(4)) we have $M=M'$ and 
 $\tilde\iota_\mathcal M=\mathrm{id}_M$.
 Since $\iota_\mathcal M$ is an open embedding, this implies 
 $\iota_\mathcal M=\mathrm{id}_\mathcal M$ and $\mathcal M'=\mathcal M$.
\end{proof}

If the assumption on the simply-connectedness of $\mathcal G$ is dropped in the 
above theorem, there exist counterexamples to
the statement, see e.g. Example~\ref{ex: S^1-example}.
Also, as e.g. illustrated in Example~\ref{ex: C-example}, it is not enough for 
$\tilde\lambda$ to be globalizable.
We really need that $\tilde\lambda$ is global.

\begin{defi}
 Let $\lambda:\g\rightarrow\mathrm{Vec}(\mathcal M)$ an infinitesimal action.
 The set of points $p\in M$ such that there exists an even vector field 
 $X\in \g_0$ with $\lambda(X)(p)\neq 0$ is called the support of $\lambda$.
\end{defi}

\begin{rmk}
 The definition of the support of an infinitesimal action~$\lambda$ implies that 
the support of $\lambda$ coincides 
 with the support of the underlying infinitesimal action~$\tilde\lambda$.
\end{rmk}

In the classical case, we have the following two theorems 
on actions of simply-connected Lie groups.

\begin{thm}[see \cite{Palais}, Chapter III, Theorem XVIII]
 Let $G$ be a simply-connected Lie group and 
 $\tilde\lambda:\g_0\rightarrow\mathrm{Vec}(M)$ an infinitesimal
 action of $G$ on a manifold $M$. If the support of $\tilde\lambda$ is 
relatively compact in $M$,
 then $\tilde\lambda$ is global.
 
 In particular,
 any infinitesimal action of a simply-connected Lie group $G$ on a compact 
manifold $M$ is global.
\end{thm}

\begin{thm}[see \cite{Palais}, Chapter IV, Theorem III]
 Let $\tilde\lambda:\g_0\rightarrow\mathrm{Vec}(M)$ be an infinitesimal action 
of a simply-connected Lie group $G$ on $M$. 
Suppose there exists a set of generators $\{X_i\}_{i\in I}$, $X_i\in \g_0$, of 
the Lie algebra $\g_0$ 
 such that the flow of each vector field $\tilde\lambda(X_i)$ is global. 
 Then the infinitesimal action $\tilde\lambda$ is global.
\end{thm}

Applying Theorem~\ref{thm: actions of simply-connected Lie supergroups}, these 
results in the classical case can be directly 
carried over to the case of infinitesimal actions of simply-connected Lie 
supergroups on supermanifolds.

\begin{cor}
 Let $\mathcal G$ be a simply-connected Lie supergroup and 
 $\lambda:\g\rightarrow \mathrm{Vec}(\mathcal M)$ an
 infinitesimal action whose support is relatively compact in $M$. 
 Then the infinitesimal action $\lambda$ is global.
 
 In particular, any infinitesimal action of a simply-connected Lie supergroup on 
a supermanifold
 with compact underlying manifold
 is global.
\end{cor}

\begin{cor}
 Let $\lambda:\g\rightarrow\mathrm{Vec}(\mathcal M)$ be an infinitesimal action 
of a simply-connected Lie supergroup $\mathcal G$
 such that there exists a set of generators $\{X_i\}_{i\in I}$, $X_i\in \g_0$, 
of $\g_0$ such that 
 each vector field $\tilde\lambda(X_i)$ has a global flow.
 Then the infinitesimal action $\lambda$ is global.
\end{cor}

A slightly weaker version of this corollary, in a formulation for DeWitt 
supermanifolds, has been
proven, in a different way, in \cite{Tuynman}. The assumption there is that all 
even vector fields have global flows, and not only a set of generators.

\end{document}